\documentclass[a4paper,11pt]{amsart}

\usepackage{amsmath,amsthm,amssymb,amsfonts}
\usepackage{enumerate}
\usepackage{hyperref}

\usepackage{color}

\allowdisplaybreaks[3]

\newcommand{\R}{\mathbb{R}}
\newcommand{\C}{\mathbb{C}}
\newcommand{\Z}{\mathbb{Z}}
\newcommand{\N}{\mathbb{N}}
\newcommand{\T}{\mathbb{T}}
\newcommand{\supp}{\operatorname{supp}}
\newcommand{\PV}{\operatorname{P.V.}}
\newcommand{\Dom}{\operatorname{Dom}}

\renewcommand{\Re}{\operatorname{Re}}

\newtheorem{thm}{Theorem}[section]

\newtheorem{cor}[thm]{Corollary}
\newtheorem{lem}[thm]{Lemma}

\theoremstyle{definition}
\newtheorem{defn}[thm]{Definition}
\newtheorem{rem}[thm]{Remark}

\numberwithin{equation}{section}



\begin{document}

\title[Nonlocal discrete diffusion equations, regularity and applications]
{Nonlocal discrete diffusion equations and the fractional discrete Laplacian, regularity and applications}

\author[\'O. Ciaurri, L. Roncal, P. R. Stinga, J. L. Torrea, and J. L. Varona]
{\'Oscar Ciaurri \and Luz Roncal \and Pablo Ra\'ul Stinga \and \\ Jos\'e L. Torrea \and Juan Luis Varona}

\address[\'O. Ciaurri and J. L. Varona]{%
Departamento de Matem\'aticas y Computaci\'on\\
Universidad de La Rioja\\
26006 Logro\~no, Spain}
\email{\{oscar.ciaurri,jvarona\}@unirioja.es}

\address[L. Roncal]{%
Basque Center for Applied Mathematics (BCAM)\\
Alameda de Mazarredo 14\\
48009 Bilbao, Spain}
\email{lroncal@bcamath.org}

\address[P. R. Stinga]{Department of Mathematics\\
Iowa State University\\
396 Carver Hall, Ames, IA\\
50011, USA}
\email{stinga@iastate.edu}

\address[J. L. Torrea]{Departamento de Matem\'aticas, Facultad de Ciencias\\
Universidad Aut\'onoma de Madrid\\
28049 Madrid, Spain}
\email{joseluis.torrea@uam.es}

\thanks{Research partially supported by grants MTM2015-66157-C2-1-P
and MTM2015-65888-C04-4-P, MINECO/FEDER, UE,
from Government of Spain}

\begin{abstract}
The analysis of nonlocal discrete equations driven by
fractional powers of the discrete Laplacian on a mesh of size $h>0$
\[
(-\Delta_h)^su=f,
\]
for $u,f:\Z_h\to\R$, $0<s<1$, is performed. The pointwise nonlocal formula for $(-\Delta_h)^su$ and
the nonlocal discrete mean value property for discrete $s$-harmonic functions are obtained.
We observe that a
characterization of $(-\Delta_h)^s$ as the Dirichlet-to-Neumann operator for a
semidiscrete degenerate elliptic local extension problem is valid.
Regularity properties and Schauder estimates in discrete H\"older spaces as well as
existence and uniqueness of solutions to the nonlocal Dirichlet problem are shown.
For the latter, the fractional discrete Sobolev embedding and the
fractional discrete Poincar\'e inequality are proved,
which are of independent interest. We introduce the negative power (fundamental solution)
\[
u=(-\Delta_h)^{-s}f,
\]
which can be seen as the Neumann-to-Dirichlet map
for the semidiscrete extension problem. We then prove
the discrete Hardy--Littlewood--Sobolev inequality for $(-\Delta_h)^{-s}$.

As applications, the convergence of our fractional discrete Laplacian to
the (continuous) fractional Laplacian as $h\to0$ in H\"older spaces
is analyzed. Indeed, uniform estimates for the error of the approximation
in terms of $h$ under minimal regularity assumptions are obtained. We finally prove
that solutions to the Poisson problem for the fractional Laplacian
\[
(-\Delta)^sU=F,
\]
in $\R$, can be approximated by solutions to the Dirichlet problem for our fractional
discrete Laplacian, with explicit uniform error estimates in terms of~$h$.
\end{abstract}

\maketitle

\section{Introduction and main results}
\label{Intro}

The fractional Laplacian, understood as a positive power of the classical Laplacian,
has been present for a long time in several areas of mathematics, like
potential theory, harmonic analysis, fractional calculus, functional analysis and probability
\cite{Getoor, Landkof, japoneses, StSingular}.
However, although this operator appeared in some differential equations in physics \cite{MK},
it was not until the past decade when it became a very popular object in
the field of partial differential equations. Indeed, nonlocal diffusion
equations involving fractional Laplacians have been
one of the most studied research topics in the present century.
The fractional  Laplacian on $\R^n$ is defined,
for $0<s<1$ and good enough functions $U$, as
\begin{equation}
\label{eq:principio}
(-\Delta)^sU(x) =  c_{n,s}\PV \int_{\R^n} \frac{U(x)-U(y)}{|x-y|^{n+2s}}\,dy,
\end{equation}
for $x\in\R^n$, where $c_{n,s}>0$ is an explicit constant, see \cite{Landkof, Stinga-Torrea}.
We could say that the triggers that produced the outbreak in the field
were the papers by L.~Caffarelli and L.~Silvestre \cite{Caffarelli-Silvestre}
and L.~Silvestre \cite{Silvestre-CPAM}.
Since the appearance of those works there has been
a substantial revision of a big amount of problems in differential equations where the
Laplacian is replaced by the fractional Laplacian or
more general integro-differential operators, see for example \cite{Caffarelli, Caffarelli-Stinga, Gale,
Roncal-Stinga, RS, Savin-Valdinoci-1, Savin-Valdinoci-2, Silvestre-CPAM, Stinga, Stinga-Torrea} and references
therein for models and techniques. On the other hand,
there is the basic question of approximating the continuous problems by discrete
ones. The large literature includes numerical approximations of different sorts,
see for example \cite{Acosta-Borthagaray, Bonito-Pasciak, DelTeso, Huang-Oberman, Lenzmann, Nochetto}
and references therein. The main difficulties to overcome in numerical approaches are the nonlocality
and singularity of the operator \eqref{eq:principio}.
In any case, it is expected for discrete jump models to approximate continuous jump
models in a good way as the size of the mesh goes to zero \cite{MK}, a question that we also address here.

One of our aims in this paper is to present a quite complete study of nonlocal discrete diffusion equations
involving the fractional powers of the discrete Laplacian
$$
(-\Delta_h)^su
$$
and show how they can be used to approximate solutions to the Poisson problem for the fractional Laplacian
\begin{equation}
\label{eq:PoissonProblem}
(-\Delta)^sU=F,\quad\text{in}~\R.
\end{equation}

We describe next our main results.

Along the paper we consider a mesh of fixed size $h>0$ on $\R$ given by
$\Z_h = \{ hj:j\in\Z\}$. For a function $u:\Z_h\to\R$ we use the notation
$u_j=u(hj)$
to denote the value of $u$ at the mesh point $hj\in\Z_h$.
The discrete Laplacian $\Delta_h$ on $\Z_h$ is then given by
\[
-\Delta_hu_j = -\frac{1}{h^2} (u_{j+1}-2u_{j}+u_{j-1}).
\]
For $0<s<1$, we define the \textit{fractional powers of the discrete Laplacian} $(-\Delta_h)^su$ on $\Z_h$
with the semigroup method (see \cite{Stinga, Stinga-Torrea}) as
\begin{equation}
\label{definition}
(-\Delta_h)^su_j=\frac{1}{\Gamma(-s)}\int_0^\infty\big(e^{t\Delta_h}u_j-u_j\big)\,\frac{dt}{t^{1+s}}.
\end{equation}
Here $w_j(t)=e^{t\Delta_h}u_j$ is the solution  to the semidiscrete heat equation
\begin{equation}
\label{eq:heatequation}
\begin{cases}
\partial_tw_j=\Delta_hw_j, &\text{in}~\Z_h\times(0,\infty),\\
w_j(0)=u_j, &\text{on}~\Z_h,
\end{cases}
\end{equation}
(see Section~\ref{sec:proofBasic}) and $\Gamma$ denotes the Gamma function.
As we will see, our technique provides a formula that gives the exact solution of this equation
and that is central along the paper.
This can be applied for fixed $h>0$, but obviously not for non-uniform meshes, for instance.

\begin{thm}[Pointwise nonlocal formula and limits]
\label{thm:basicProperties}
For $0\leq s\leq1$, we let
\begin{equation}
\label{ls}
\ell_{\pm s} := \Big\{u:\Z_h\to\R: \|u\|_{\ell_{\pm s}} 
:= \sum_{m\in \Z} \frac{|u_{m}|}{(1+|m|)^{1 \pm 2s}}<\infty\Big\}.
\end{equation}
\begin{enumerate}[$(a)$]
\item \label{aa} If $0<s<1$ and $u\in\ell_{s}$ then
\begin{equation}
\label{eq:puntualdiscreta}
(-\Delta_h)^su_{j} = \sum_{m\in\Z,m\neq j}
\big(u_{j}-u_{m}\big) K^h_s(j-m),
\end{equation}
where the discrete kernel $K^h_s$ is given by
\begin{equation}
\label{eq:frKernelOned}
K^h_s(m)=\frac{4^s\Gamma(1/2+s)}{\sqrt{\pi}|\Gamma(-s)|}\cdot\frac{\Gamma(|m|-s)}{h^{2s}\Gamma(|m|+1+s)},
\end{equation}
for any $m\in\Z\setminus\{0\}$, and $K_s^h(0)=0$.
\item \label{cc} For $0<s<1$ there exist constants $0<c_s\leq C_s$ such that, for any $m\in\Z\setminus\{0\}$,
\begin{equation}
\label{eq:frKernelEst}
\frac{c_s}{h^{2s}|m|^{1+2s}}\leq K^h_s(m) \leq \frac{C_{s}}{h^{2s}|m|^{1+2s}}.
\end{equation}
\item \label{bb} If $u\in\ell_{0}$ then
$\displaystyle \lim_{s\to0^+}(-\Delta_h)^su_{j}=u_{j}$.
\item  If $u$ is bounded then
$\displaystyle \lim_{s\to1^-}(-\Delta_h)^su_{j}=-\Delta_hu_{j}.$
\end{enumerate}
\end{thm}

The expression in \eqref{eq:puntualdiscreta} and the estimate in \eqref{eq:frKernelEst} show that
the fractional discrete Laplacian is a nonlocal operator on $\Z_h$ of order $2s$
(we precise this in Theorems~\ref{thm:Holder}~and~\ref{RegularityfractionalIntegral}).
Notice also that our definition \eqref{definition} is neither
a direct discretization of the pointwise formula for the fractional Laplacian \eqref{eq:principio}, nor
a ``discrete analogue'', but the $s$-fractional power of the discrete Laplacian.
In this regard, we warn the reader that the notation $(-\Delta_h)^{\alpha/2}$, $0<\alpha<2$,
used in \cite{Huang-Oberman}
does not refer to the fractional power of the discrete Laplacian \eqref{eq:puntualdiscreta}, but
to a specific way of discretizing the pointwise formula in \eqref{eq:principio}.
The constant
\begin{equation}
\label{eq:constante1d}
A_{s} := \frac{4^s\Gamma(1/2+s)}{\sqrt{\pi}|\Gamma(-s)|}>0,
\end{equation}
which appears in the kernel $K_s^h(m)$, see \eqref{eq:frKernelOned},
is exactly the same constant $c_{n,s}>0$ in the formula for
the fractional Laplacian \eqref{eq:principio} when $n=1$.

\begin{rem}[Mean value formula and probabilistic interpretation]
Let $u$ be a discrete harmonic function on $\Z_h$, that is,
$-\Delta_hu=0$. This is equivalent as saying that $u$ satisfies the discrete mean value property:
\[
u_{j} = \frac12u_{j+1} +\frac12 u_{j-1}.
\]
This identity shows that a
discrete harmonic function describes the random movement of a particle that jumps either to the adjacent left
point or to the adjacent right point with probability~$1/2$.
Suppose now that $u$ is a fractional discrete harmonic function, that is, $(-\Delta_h)^su_j=0$. Then
from \eqref{eq:puntualdiscreta} we have the following nonlocal mean value property:
\[
u_j=\frac{1}{\Sigma^h_s}\sum_{m\in\Z,m\neq j}u_{m}K^h_s(j-m)=:\sum_{m\in\Z}u_mP_s(j-m),
\]
where $\Sigma^h_s:=\sum_{m\in\Z}K^h_s(m)=A_sh^{-2s}/s=
\frac{4^s\Gamma(1/2+s)}{h^{2s}\sqrt{\pi}\,\Gamma(1+s)}$,
so that $P_s(m)$ is a probability density on $\Z$ with $P_s(0)=0$ which is \textit{independent of $h>0$}.
In a parallel way we understand this last identity by saying that a
fractional discrete harmonic function describes a particle
that is allowed to randomly jump to any point on $\Z_h$ (not only to the adjacent ones)
and that the probability to jump from the point $hj$ to the point $hm$ is $P_s(j-m)$.
By \eqref{eq:frKernelEst} this probability is proportional to $|j-m|^{-(1+2s)}$.
As $s\to1^-$ the probability to jump from $j$ to a non adjacent point tends to zero, while the probability
to jump to an adjacent point tends to one, recovering in this way the previous situation. As $s\to0^+$,
the probability to jump to any point tends to zero, so there are no jumps.
\end{rem}

The solution to the fractional discrete Poisson problem $(-\Delta_h)^su=f$ in $\Z_h$ is
realized by the \textit{negative powers of the discrete Laplacian},
which are also called the \textit{fractional discrete integrals}.
They are defined, for $s>0$ and a function $f:\Z_h\to\R$, as
\begin{equation}
\label{eq:fractionalintegralsemigrupo}
(-\Delta_h)^{-s}f_{j} = \frac{1}{\Gamma(s)} \int_0^{\infty}e^{t\Delta_h}f_{j}\,\frac{dt}{t^{1-s}}.
\end{equation}
The kernel of $(-\Delta_h)^{-s}$ is the fundamental solution of $(-\Delta_h)^s$.

\begin{thm}[Fundamental solution and Hardy--Littlewood--Sobolev inequality]\label{lem:kernelFractionalIntegral}
Let us fix $0<s<1/2$ and let $f\in\ell_{-s}$ (see \eqref{ls}).
\begin{enumerate}[$(a)$]
\item We have the pointwise formula
\begin{equation}
\label{formulapuntualpotenciasnegativas}
(-\Delta_h)^{-s}f_j=\sum_{m\in\Z}K^h_{-s}(j-m)f_{m},
\end{equation}
where, for $m\in\Z$, the discrete kernel $K^h_{-s}(m)$ is given by
\begin{equation}
\label{eq:kernelFrIntegralOned}
K_{-s}^h(m)=\frac{4^{-s}\Gamma(1/2-s)}{\sqrt{\pi}\,\Gamma(s)}\cdot\frac{\Gamma(|m|+s)}{h^{-2s}\Gamma(|m|+1-s)}.
\end{equation}
\item There exist positive constants  $c_s$, $C_s$ and $d_s\leq D_s$
such that, for $m\in\Z\setminus\{0\}$,
\begin{equation}
\label{eq:growthFractional0}
\frac{d_s}{h^{-2s}|m|^{1-2s}}\leq K_{-s}^h(m)\leq\frac{D_s}{h^{-2s}|m|^{1-2s}},
\end{equation}
and
\begin{equation}
\label{eq:growthFractional}
\bigg| K_{-s}^h(m) - \frac{c_{s}}{h^{-2s}|m|^{1-2s}}\bigg| \le \frac{C_{s}}{h^{-2s}|m|^{2-2s}}.
\end{equation}
\item Let $1<p<q<\infty$ such that $1/q\leq1/p-2s$.
There exists a constant $C>0$, depending only on $p$, $q$ and $s$, such that if $f\in\ell^p_h$
(see \eqref{Lph})
then $(-\Delta_h)^{-s}f\in\ell^q_h$ and
\begin{equation}
\label{HLS}
\|(-\Delta_h)^{-s}f\|_{\ell^q_h} \leq\frac{C}{h^{1/p-2s-1/q}}\|f\|_{\ell^p_h}.
\end{equation}
\end{enumerate}
\end{thm}

It is worth comparing formula \eqref{eq:kernelFrIntegralOned} for the
kernel of the fractional discrete integral $K_{-s}^h(m)$
with the one for the kernel of the fractional discrete Laplacian $K_s^h(m)$ in \eqref{eq:frKernelOned}.
We also point out that the factor of $h$ disappears from the right hand side of \eqref{HLS}
when we reach the critical exponent $q=p/(1-2sp)$. As before, \eqref{formulapuntualpotenciasnegativas}
is the $(-s)$-power of the discrete Laplacian, not a ``discrete analogue'' as that of \cite{Stein-Wainger}.
Observe that the constant
\begin{equation}
\label{constantfractionalintegral}
A_{-s}:=\frac{4^{-s}\Gamma(1/2-s)}{\sqrt{\pi}\Gamma(s)}>0,
\end{equation}
appearing in the kernel $K_{-s}^h$, see \eqref{eq:kernelFrIntegralOned},
is exactly the same normalizing constant for the fractional integral
$(-\Delta)^{-s}$ in dimension one in which $0<s<n/2$, $n=1$,
see \cite{StSingular} and Theorem \ref{thm:continuousPoissonproblem}.

\begin{rem}[Extension problem for $(-\Delta_h)^s$ and $(-\Delta_h)^{-s}$]
\label{rem:extension problem}
The fractional powers of the discrete Laplacian
which, as we have seen in Theorems \ref{thm:basicProperties} and \ref{lem:kernelFractionalIntegral},
are nonlocal discrete operators, can be regarded as
boundary values (Dirichlet or Neumann) of the solution to
a local semidiscrete elliptic extension problem. This observation is just an application of the
general extension problem of \cite{Stinga, Stinga-Torrea}, see also \cite{Gale}.
Thus, the positive powers can be seen as Dirichlet-to-Neumann maps,
while the negative ones as Neumann-to-Dirichlet maps.
Indeed, consider the semidiscrete degenerate elliptic operator
$$
L_{a,h}= \Delta_h+\tfrac{a}{y}\partial_y+\partial_{yy},
$$
where $a=1-2s$ and $0<s<1$.
This operator acts on semidiscrete functions $w=w_j(y)=w(hj,y):\Z_h\times(0,\infty)\to\R$.
Given $u\in\Dom((-\Delta_h)^s)$, the semidiscrete function $w$ defined as
$$
w_j(y)= \frac{y^{2s}}{4^s \Gamma(s)} \int_0^\infty e^{-y^2/(4t)}e^{t\Delta_h}u_j\,\frac{dt}{t^{1+s}}
$$
is the unique solution (weakly vanishing as $y\to\infty$) to the Dirichlet problem
$$
\begin{cases}
L_{a,h}w=0,&\text{in } \Z_h\times(0,\infty),\\
w_j(0)=u_j,&\text{on } \Z_h.
\end{cases}
$$
Moreover,
$$
-\lim_{y\rightarrow 0^+}y^a\partial_yw_j(y)=
-2s\lim_{y\to0^+}\frac{w_j(y)-w_j(0)}{y^{2s}}=
\frac{\Gamma(1-s)}{4^{s-1/2}\Gamma(s)}(-\Delta_h)^su_j.
$$
Analogously, given $f\in\Dom((-\Delta_h)^{-s})$, the semidiscrete function $v$ defined as
$$
v_j(y)= \frac{1}{\Gamma(s)}\int_0^\infty e^{-y^2/(4t)}e^{t\Delta_h}f_j\,\frac{dt}{t^{1-s}},
$$
is the unique solution (weakly vanishing as $y\to\infty$) to the Neumann problem
$$
\begin{cases}
L_{a,h}v=0 &\text{in } \Z_h\times(0,\infty),\\
-y^a\partial_yv_j(y)\big|_{y=0^+}=f_j,&\text{on } \Z_h.
\end{cases}
$$
Moreover,
$$
\lim_{y\rightarrow 0^+}v_j(y)=\frac{4^{s-1/2}\Gamma(s)}{\Gamma(1-s)}(-\Delta_h)^{-s}f_j.
$$
It is obvious that if we have $(-\Delta_h)^su=\frac{4^{s-1/2}\Gamma(s)}{\Gamma(1-s)}f$ then $w=v$.
\end{rem}

We next go back to the fractional discrete Laplacian and show that it behaves
as a fractional discrete derivative of order $2s$ in discrete H\"older spaces.
This will be obtained by exploiting \eqref{eq:puntualdiscreta}. The following estimates
are parallel to the corresponding ones for the fractional Laplacian
(see \cite{Silvestre-CPAM}). For the definition of discrete H\"older spaces $C^{k,\alpha}_h$
see Definition~\ref{def:discreteH}.

\begin{thm}[Fractional discrete Laplacian in discrete H\"older spaces]
\label{thm:Holder}
Let $k\geq0$, $0<\alpha\le1$, $0<s<1$ and $u\in\ell_s$ (see \eqref{ls}).
\begin{enumerate}[$(i)$]
\item \label{(i)} If $u\in C^{k,\alpha}_h$ and $2s<\alpha$ then
$(-\Delta_h)^su\in C^{k,\alpha-2s}_h$ and
$$
[(-\Delta_h)^su]_{C^{k,\alpha-2s}_h}\leq C[u]_{C^{k,\alpha}_h}.
$$
\item \label{(iii)} If $u\in C^{k+1,\alpha}_h$ and $2s>\alpha$ then
$(-\Delta_h)^su\in C^{k,\alpha-2s+1}_h$ and
$$
[(-\Delta_h)^su]_{C^{k,\alpha-2s+1}_h}\leq C[u]_{C^{k+1,\alpha}_h}.
$$
\end{enumerate}
The constants $C>0$ appearing above are independent of $h>0$ and~$u$.
\end{thm}

The following result, which complements Theorem~\ref{thm:Holder},
contains the discrete Schauder estimates for the fractional discrete Laplacian.

\begin{thm}[Discrete Schauder estimates]
\label{RegularityfractionalIntegral}
Let $k\geq0$, $0<\alpha\leq1$, $0<s<1/2$ and $f\in \ell_{-s}$ (see~\eqref{ls}).
\begin{enumerate}[$(i)$]
\item If $f\in C_h^{k, \alpha} $ and $2s+\alpha <1$ then $(-\Delta_h)^{-s} f \in C^{k, \alpha+2s}_h$ and
$$
[ (-\Delta_h)^{-s} f ]_{C_h^{k,\alpha+2s}} \le C [ f]_{C_h^{k,\alpha}}.
$$
\item If $f\in C_h^{k, \alpha}$ and $2s+\alpha >1$ then $(-\Delta_h)^{-s} f \in C_h^{k+1, \alpha+2s-1}$ and
$$
[ (-\Delta_h)^{-s} f ]_{C_h^{k+1,\alpha+2s-1}} \le C [ f]_{C_h^{k,\alpha}}.
$$
\item If $f\in \ell^\infty_h$, see \eqref{Linftyh}, then $(-\Delta_h)^{-s} f \in C^{0,2s}_h$ and
$$
[ (-\Delta_h)^{-s} f ]_{C_h^{0,2s}} \le C \|f\|_{\ell^\infty_h}.
$$
\end{enumerate}
The constants $C>0$ appearing above are independent of $h>0$ and~$f$.
\end{thm}

Next we present what might be considered the most interesting results of this paper.
We show how the fractional discrete Laplacian approximates
the fractional Laplacian as $h\to0$ in the strongest possible sense, that is, in the uniform norm.
We need some notation.
Given a function $U=U(x):\R\to\R$, we define its restriction $r_hU:\Z_h\to\R$
to the mesh $\Z_h$ to be the discrete function (or sequence) $(r_hU)_j:=U(hj)$, for $hj\in\Z_h$.
The first approximation result considers uniform estimates for differences of the type
\[
\big\|(-\Delta_h)^s(r_hU)-r_h \big((-\Delta)^s U\big) \big\|_{\ell^\infty_h}
\]
in terms of the size $h$ of the mesh.
The estimates will certainly depend on the regularity of $U$,
which we take to be in a H\"older space $C^{k,\alpha}$
(see Definition~\ref{def:contH}).
The notation $D_+u$ refers to the discrete derivative of $u:\Z_h\to\R$,
see~\eqref{eq:definitiondiscretederivatives}.

\begin{thm}[Uniform comparison with fractional Laplacian]
\label{thm:consistencia1d}
Let $0<\alpha\le 1$, $0<s<1$.
\begin{enumerate}[$(i)$]
\item \label{(iCon)} If $U\in C^{0,\alpha}$ and $2s<\alpha$ then
\[
\|\, (-\Delta_h)^s(r_hU)-r_h((-\Delta)^sU)\, \|_{\ell_h^\infty}\leq C[U]_{C^{0,\alpha}}h^{\alpha-2s}.
\]
\item \label{(iiCon)} If $U\in C^{1,\alpha}$ and $2s<\alpha$ then
\[
\|D_+(-\Delta_h)^s(r_hU)-r_h(\tfrac{d}{dx}(-\Delta)^sU)\|_{\ell_h^\infty}\leq C[U]_{C^{1,\alpha}}h^{\alpha-2s}.
\]
\item \label{(iiiCon)} If $U\in C^{1,\alpha}$ and $\alpha<2s<1+\alpha$ then
\[
\|(-\Delta_h)^s(r_hU)-r_h((-\Delta)^sU)\|_{\ell_h^\infty}\leq C[U]_{C^{1,\alpha}}h^{\alpha-2s+1}.
\]
\item \label{(ivCon)} If $U\in C^{k,\alpha}$ and $k+\alpha-2s$
is not an integer then
\[
\|D_+^l(-\Delta_h)^s(r_hU)-r_h(\tfrac{d^{l}}{dx^{l}}(-\Delta)^sU)\|_{\ell_h^\infty}
\leq C[U]_{C^{k,\alpha}}h^{\alpha-2s+k-l},
\]
where $l$ is the integer part of $k+\alpha-2s$.
\end{enumerate}
The constants $C>0$ appearing above are independent of $h$ and~$U$.
\end{thm}

Although the proof of Theorem~\ref{thm:consistencia1d} is not trivial,
one could say in a very na\"ive way that such a result
is in some sense announced by Theorem~\ref{thm:Holder}. Indeed,
the fractional discrete Laplacian maps $C^{\beta}_h$ into $C^{\beta-2s}_h$. The continuous version
of this property is also true for the fractional Laplacian, so the restriction of $(-\Delta)^sU$
to the mesh $\Z_h$ is in $C^{\beta-2s}_h$
whenever $U\in C^{\beta}$.
We also point out that $D_+$ in Theorem~\ref{thm:consistencia1d}
can be replaced by $D_-$, see~\eqref{eq:definitiondiscretederivatives}.

The second approximation statement
is the convergence of discrete solutions to continuous ones:
the solution to the Poisson problem for the fractional Laplacian
\eqref{eq:PoissonProblem} can be approximated by using the solution
to the Dirichlet problem for the fractional discrete Laplacian \eqref{ecdiscreta}.
For $R>0$, we set $B_R^h=\{hj\in \Z_h: |hj|<R\}$
and $B_R=(-R,R)\subset\R$.

\begin{thm}[Convergence of discrete solutions to continuous ones]
\label{thm:Convergence-result}
Let $0<\alpha,s<1$ such that $\alpha+2s<1$.
Let $F\in C^{0,\alpha}$ with compact support contained in an interval $B_{R_0}$, for some $R_0>0$.
Let $U\in C^{0,\alpha+2s}$ be the unique solution
to the Poisson problem \eqref{eq:PoissonProblem}
vanishing at infinity (see Theorem~\ref{thm:continuousPoissonproblem}
where this function $U$ is explicitly constructed).
Fix $h>0$ and let $f=r_hF$, the restriction of $F$ to~$\Z_h$. Let $u:\Z_h\to\R$ be
the unique solution (provided by Theorem~\ref{thm:discreteDirichletproblem}) to the discrete Dirichlet problem
\begin{equation}
\label{ecdiscreta}
\begin{cases}
(-\Delta_h)^su=f, & \text{ in } B_R^h, \\
u=0, & \text{ in } \Z_h\setminus B_R^h,
\end{cases}
\end{equation}
where $R>\max\{2R_0,h^{-\alpha}\}$. Then there is a constant $C>0$,
depending on $s$, $\alpha$ and $R_0$, but not on $R$ or $h$, such that
\begin{equation}
\label{LAestimacion}
\|u-r_hU\|_{\ell^\infty_h(B_R^h)}\leq C\|F\|_{C^{0,\alpha}}R^{2s}h^{\alpha}.
\end{equation}
\end{thm}

As far as the authors are aware, Theorem~\ref{thm:Convergence-result} is the first result where
error estimates in the $L^\infty$-norm for approximations of solutions to the Poisson problem
for the fractional Laplacian by a nonlocal discrete problem are proved.
We stress that the Poisson problem $(-\Delta)^sU=F$ in H\"older spaces is non variational
and $U$ and $u$ are classical solutions.
In the discrete problem \eqref{ecdiscreta} a Dirichlet boundary condition
means to prescribe the values of $u$ outside $B_R^h$ because
of the nonlocality of the fractional discrete Laplacian, see \eqref{eq:puntualdiscreta}.

On the part of the domain that is left out of the estimate \eqref{LAestimacion},
that is, outside of $B_R^h$, we are approximating $U$ by the zero function so,
in particular,
$$
\|u-r_hU\|_{\ell^\infty_h(\Z_h\setminus B_R^h)}=\|r_hU\|_{\ell^\infty_h(\Z_h\setminus B_R^h)}
\leq C\frac{\|F\|_{L^\infty}}{R^{1-2s}},
$$
where $C>0$ depends only on $s$, see \eqref{laestimaciondeloqueresta} and
also Theorem~\ref{thm:continuousPoissonproblem}. This estimate
is sharp. Indeed, suppose that
$$
\chi_{[-1,1]}(x)\leq F(x)\leq\chi_{[-2,2]}(x),\quad\text{for every } x\in\R.
$$
Then, for any $x>4$, in the notation of Theorem 9.9,
$$
U(x) = A_{-s}\int_{\R}\frac{F(y)}{|x-y|^{1-2s}}\,dy
\geq A_{-s}\int_{0}^1\frac{1}{(x-y)^{1-2s}}\,dy\geq \frac{C_s}{x^{1-2s}}.
$$

Another important fact we state in Theorem~\ref{thm:Convergence-result}
is the unique solvability of the discrete Dirichlet problem \eqref{ecdiscreta}.
We show this in Section~\ref{Section:Dirichlet},
see Theorem~\ref{thm:discreteDirichletproblem}.
On the way we need to prove the fractional discrete Sobolev embedding
$$
\|u\|_{\ell_h^{2/(1-2s)}}\leq C_s\|(-\Delta_h)^{s/2}u\|_{\ell^2_h}
$$
(in which $s<1/2$)
and obtain as a consequence, see Theorem~\ref{thm:SobolevandPoincare}, the fractional discrete Poincar\'e
inequality
$$
\|u\|_{\ell_h^{2}}\leq C_sh^s\big(\#_h\supp(u)\big)^s\|(-\Delta_h)^{s/2}u\|_{\ell^2_h},
$$
where $\#_hE$ denotes the number of points in the set $E$.
The proof of the latter inequality
is postponed until Subsection~\ref{subsection:Sobolev}.

Theorem~\ref{thm:Convergence-result}
will then be a consequence of Theorem~\ref{thm:Holder}$(i)$ and
the nonlocal discrete maximum principle we prove in
Section~\ref{sec:Dirichlet}, see Theorem~\ref{thm:DiscretePoisson1d}.
The presence of the factor $R^{2s}$ in \eqref{LAestimacion} is actually
natural in view of such a maximum principle.

We also claim in Theorem~\ref{thm:Convergence-result}
the existence of a unique, explicitly computed, classical solution $U$ to the Poisson problem~\eqref{eq:PoissonProblem}.
Though we believe this statement belongs to the folklore, we will present a self contained
proof showing that such solution is indeed $U(x)=(-\Delta)^{-s}F(x)$,
see Theorem~\ref{thm:continuousPoissonproblem}. Note that $(-\Delta)^{-s}$
defines a tempered distribution on $\R$ if and only if $s<1/2$, see \cite{Silvestre-CPAM}.
In addition, both the minimal regularity hypothesis $0<\alpha+2s<1$ 
and the explicit formula for $U$ in Theorem~\ref{thm:continuousPoissonproblem}
in terms of $(-\Delta)^{-s}$, as well as
the assumption $2s<\alpha$ in Theorem~\ref{thm:consistencia1d}$(i)$, that are used in the proof of Theorem~\ref{thm:Convergence-result}, imply $0<s<1/2$. The same range of $s$
is considered in \cite{Savin-Valdinoci-1} and \cite{Silvestre-CPAM}. On the other hand $U$ has the minimal regularity.
Hence no extra smoothness other than the correct one is assumed.

In an extremely na\"ive way, and parallel to Theorem~\ref{thm:consistencia1d},  one could think that estimate \eqref{LAestimacion} should be suggested by the value of the difference 
$$
\|(-\Delta_h)^{-s}(r_hF)-r_h((-\Delta)^{-s}F)\, \|_{\ell_h^\infty}. 
$$ 
However, in this case the situation is completely different. This is due to the absence of information about $(-\Delta_h)^s u $ in $\Z_h\setminus B_R^h$. Even more, in our Theorem~\ref{thm:Convergence-result} the presence of $R$
is essential as we showed a few lines above. Of course, different types of discrete problems could be chosen to approximate the Poisson problem \eqref{eq:PoissonProblem}. Our results could also raise the question of the approximation of the solution of the Dirichlet problem
$$
(-\Delta)^sV = G   \text{ in } B_R,  \quad  V=0 \text{ in }  \R \setminus  B_R,
$$
but we observe that, in this case, a preliminary discussion about the appropriate
definition of $(-\Delta)^s$ that is being used (regional, restriction of the global, etc) should be given.

It is important to notice that the optimal H\"older regularity of the solution
$u$ to the discrete Dirichlet problem \eqref{ecdiscreta} is not known.
We conjecture that $u\in C^{0,s}_h$. In any case, $u$,
being a sequence with only finitely many nonzero terms,
is a classical solution and our error estimate \eqref{LAestimacion} is explicit in terms of $h$ and~$R$.
Observe that, as $h$ tends to zero, the solution $u$ in \eqref{ecdiscreta}
must be found in a larger domain~$B_R^h$.

One of the main strategies used to obtain our results is the method of semigroups. Since the
semidiscrete heat semigroup is given in terms of modified Bessel functions, see Section~\ref{sec:proofBasic},
we will exhaustively use some properties and facts about these functions that we collect
in Subsection~\ref{sec:preliminaries}.

Some of our results can be easily extended
to higher dimensions (for example, the extension problem in Remark \ref{rem:extension problem})
and we leave this task to the interested reader.
In fact, it is possible to define the multi-dimensional discrete Laplacian and explicitly write down the
solution to the corresponding semidiscrete heat equation (which is
a key tool along our paper) as a discrete convolution
with a heat kernel given as a product of Bessel functions of different integer orders.
In this paper we present several fine results involving precise
estimates of Bessel and Gamma functions that appear to be quite
non trivial to mimic in higher dimensions.
The semigroup method we use is obviously independent of the dimension,
so semigroup formulas for fractional powers of multidimensional discrete Laplacians
can be written down.
However, closed pointwise formulas as explicit as the ones we discovered in
Theorems~\ref{thm:basicProperties}~and~\ref{lem:kernelFractionalIntegral},
that are crucial along the paper, will not be available anymore. Thus a different method must be found
to prove our results in higher dimensions and we pose this together with 
the H\"older regularity of the solution $u$ to \eqref{ecdiscreta}
as open problems. In any case, not only our main results are certainly novel,
but also our techniques, which involve the manipulation of the semidiscrete heat equation.
We mention related questions raised in \cite{japoneses, Zoia}. Those works deal
with discrete Laplacians only at the level of $L^2$-spaces. Instead, we work with H\"older
spaces, presenting estimates in the uniform norm (in this regard, our paper does not deal
with variational problems and techniques) with explicit dependence on $h$.
As a matter of fact, all our estimates recover the continuous ones as $h\to0^+$.

The structure of the paper is as follows.
Section~\ref{sec:proofBasic} is devoted to the proof of Theorem~\ref{thm:basicProperties}.
In Section~\ref{sec:dicreteFractional} we prove Theorem~\ref{lem:kernelFractionalIntegral}.
The proofs of Theorems~\ref{thm:Holder}~and~\ref{thm:consistencia1d} are
presented in Section~\ref{Section:dosteoremas}.
Section~\ref{Section:JoseLuis} contains the proof of Theorem~\ref{RegularityfractionalIntegral}.
The Dirichlet problem for the fractional discrete Laplacian is analyzed
in Section~\ref{Section:Dirichlet}. Section~\ref{sec:Dirichlet} contains the discrete maximum principle.
The proof of Theorem~\ref{thm:Convergence-result} is done in Section~\ref{sec:proofmain}.
Some technical lemmas, the properties of Bessel functions,
the proofs of the fractional discrete Sobolev and Poincar\'e inequalities
and the analysis of the Poisson problem \eqref{eq:PoissonProblem} are
all collected in Section~\ref{sec:technical}.
By $C_s,c_s,D_s,d_s$ we mean positive constants depending on $s$ that may change in each occurrence, while
by $C$ we will denote a constant independent of the significant variables.
The notation $B^h$ refers to generic discrete finite interval contained in~$\Z_h$.

\section{Proof of Theorem~\ref{thm:basicProperties}}
\label{sec:proofBasic}

Given $u:\Z_h\to\R$, the solution to the semidiscrete heat equation \eqref{eq:heatequation} can be written as
\begin{equation}
\label{semigrupo}
e^{t\Delta_h}u_{j} = \sum_{m\in \Z}G\big(j-m,\tfrac{t}{h^2}\big)u_{m}=
\sum_{m\in \Z}G\big(m,\tfrac{t}{h^2}\big)u_{j-m}, \quad t\geq0,
\end{equation}
where the semidiscrete heat kernel $G$ is defined as
\begin{equation}
\label{eq:semidiscreteheatkernel}
G(m,t) = e^{-2t} I_{m}(2t),\quad m\in \Z,~t\geq0.
\end{equation}
Here $I_\nu$ is the modified Bessel function of order $\nu$ that satisfies 
\[
  \frac{\partial}{\partial t} I_k(t) = \frac{1}{2} (I_{k+1}(t) + I_{k-1}(t)),
\]
and from this we have  immediately
\begin{equation}
\label{eq:derivadaCalort}
  \frac{\partial}{\partial t} (e^{-2t} I_k(2t)) = e^{-2t} (I_{k+1}(2t) - 2I_{k}(2t) + I_{k-1}(2t)).
\end{equation}
Then formula \eqref{semigrupo} for $h=1$ follows from \eqref{eq:derivadaCalort}.
For a fixed $h\neq 1$, the formula \eqref{semigrupo} follows  by scaling. See also \cite{FiveGuys,GrIl}.
By \eqref{eq:negPos} and \eqref{eq:Ik>0} the kernel $G(m,t)$ is
symmetric in $m$, that is, $G(m,t)=G(-m,t)$, and positive.

Let us begin now with the proof of Theorem~\ref{thm:basicProperties}.

First we check that if $u\in\ell_{s}$, $0\leq s\leq1$, then
$e^{t\Delta_h}u_{j}$ is well defined. Indeed, if $N>0$, for fixed $t,h>0$,
by using the asymptotic of the Bessel function for large order~\eqref{eq:asymptotics-order-large},
\begin{equation}
\label{eq:cuentita}
\begin{aligned}
\sum_{|m|>N}G\big(m,\tfrac{t}{h^2}\big)|u_{j-m}| &\le
Ce^{-2t/h^2} \sum_{|m|>N} \frac{(et/h^2)^{|m|}(1+|m-j|)^{1+2s}}
{|m|^{|m|+1/2}} \frac{|u_{m-j}|}{(1+|m-j|)^{1+2s}}\\
&\leq Ce^{-2t/h^2} \sup_{ |m|>N}\frac{(et/h^2)^{|m|}(1+|m|+|j|)^{1+2s}}
{|m|^{|m|+1/2}}\|u\|_{\ell_s}\\
&= C_{t,h,s,N,j}\|u\|_{\ell_s}<\infty.
\end{aligned}
\end{equation}

Next we prove each of the items of the statement of Theorem~\ref{thm:basicProperties}.

\noindent$(a)$. Define
\begin{equation}
\label{eq:kernelFrLaplacian}
K_s^h(m) = \frac1{|\Gamma(-s)|} \int_0^\infty G(m,\tfrac{t}{h^2})\,\frac{dt}{t^{1+s}}
=\frac{1}{h^{2s}|\Gamma(-s)|}\int_0^\infty G(m,r)\,\frac{dr}{r^{1+s}},
\end{equation}
for $m\neq 0$, and $K^h_s(0) =0$. The symmetry of this kernel in $m$ follows from the symmetry of $G(m,t)$.
Therefore it is enough to assume that $m\in \N$. To get formula~\eqref{eq:frKernelOned},
we use \eqref{eq:rusos} with $c=2$ and $\nu=m$. On the other hand,
it is easy to show that $e^{t\Delta_h}1\equiv1$
(see for example \cite{FiveGuys} for the case $h=1$). Hence, from \eqref{definition} and~\eqref{semigrupo},
\begin{align*}
(-\Delta_h)^su_{j}
&= \frac{1}{\Gamma(-s)} \int_0^\infty \sum_{m\neq j}
G\big(j-m,\tfrac{t}{h^2}\big)(u_{m}-u_{j}) \frac{dt}{t^{1+s}} \\
&= \frac{1}{\Gamma(-s)} \sum_{m\neq j} (u_{m}-u_{j})\int_0^\infty
G\big(j-m,\tfrac{t}{h^2}\big) \frac{dt}{t^{1+s}} \\
&= \sum_{m\neq j}(u_{j}-u_{m})K^h_{s}(j-m).
\end{align*}
For the interchange of summation and integration in the second equality, consider the terms
\[
\int_0^\infty \sum_{m\neq j}
G\big(j-m,\tfrac{t}{h^2}\big)|u_{m}| \frac{dt}{t^{1+s}}
+|u_{j}| \int_0^\infty \sum_{m\neq j}
G\big(j-m,\tfrac{t}{h^2}\big)\frac{dt}{t^{1+s}}.
\]
By using \eqref{eq:frKernelEst} we see that
the first term above is bounded
by $C_{s,h}\sum_{m\neq j}|m-j|^{-(1+2s)}|u_{m}|$, which is finite for each $j$
because $u\in \ell_{s}$. For the second term we use again~\eqref{eq:frKernelEst}.

\noindent$(b)$. The two sided estimate in \eqref{eq:frKernelEst} follows from
the explicit formula for the kernel \eqref{eq:frKernelOned} and the properties of the Gamma function
we prove in Subsection~\ref{subsec:technical-lemmas}, Lemma~\ref{eq:lowerBound}.

\noindent$(c)$. Observe that
\begin{equation}
\label{relacionhy1}
h^{2s}K^h_s(m)=K_s^1(m),\quad m\neq0.
\end{equation}
We have
\[
h^{2s}(-\Delta_h)^su_{j} =
u_{j}\sum_{m\neq j} K^1_s(j-m) -\sum_{m\neq j}K^1_s(j-m)u_{m} =: u_{j}T_1-T_2.
\]
We write $T_1=T_{1,1}+T_{1,2}$ (see \eqref{eq:kernelFrLaplacian}), where
\[
T_{1,2} = \frac{1}{|\Gamma(-s)|} \sum_{m\neq j} \int_1^\infty G(j-m,t)\,\frac{dt}{t^{1+s}}=
\frac{1}{|\Gamma(-s)|} \sum_{m\neq 0} \int_1^\infty G(m,t)\,\frac{dt}{t^{1+s}}.
\]
We are going to prove that $T_{1,1}$ and $T_2$ tend to zero, while $T_{1,2}$ tends to 1,
as $s\to0^+$. Let us begin with $T_{1,2}$. By adding and subtracting the term $m=0$ in the sum
and using~\eqref{eq:sumIk}, we get
\[
T_{1,2} = \frac{1}{|\Gamma(-s)|}\left(\frac{1}{s}-\int_1^\infty \frac{e^{-2t} {I_0(2t)}}{t^{1+s}}\,dt\right).
\]
By noticing that $|\Gamma(-s)|s=\Gamma(1-s)$ and that, by~\eqref{eq:asymptotics-infinite}, we have
\[
\frac{1}{|\Gamma(-s)|} \int_1^\infty \frac{e^{-2t} {I_0(2t)}}{t^{1+s}}\,dt
\le \frac{C}{|\Gamma(-s)|} \int_{1}^{\infty}t^{-1/2-1-s}\,dt
= \frac{C}{|\Gamma(-s)|(1/2+s)},
\]
we get $T_{1,2}\to 1$ as $s\to0^+$, as desired. Next we handle the other two terms $T_{1,1}$ and $T_2$.
On one hand, by~\eqref{eq:asymptotics-zero-ctes},
$$
T_{1,1}\sim \frac{1}{|\Gamma(-s)|} \sum_{m\neq 0}
\frac{1}{\Gamma(|m|+1)} \int_0^1e^{-2t}t^{|m|}
\,\frac{dt}{t^{1+s}}
\le \frac{1}{|\Gamma(-s)|}
\sum_{m\neq 0} \frac{1}{\Gamma(|m|+1)} \frac{1}{|m|-s},
$$
and the last quantity tends to $0$ as $s\to 0^+$. On the other hand,
for $T_2$, we use~\eqref{eq:frKernelOned} to obtain
\[
|T_2|\le \frac{C_{s}}{|\Gamma(-s)|} \sum_{m\neq j}
\frac{\Gamma(|j-m|-s)}{\Gamma(|j-m|+1+s)}|u_{m}|.
\]
The constant $C_{s}$ remains bounded as $s\to0^+$.
Since $u\in \ell_0$, by dominated convergence, the sum
above is bounded by $\|u\|_{\ell_0}$, for each $j\in \Z$, as $s\to0^+$.
Therefore $T_2\to0$ as $s\to 0^+$.

\noindent$(d)$. By using the symmetry of the kernel $K_s^h$ we can write
\[
h^{2s}(-\Delta_h)^s u_{j} = S_1+S_2,
\]
where (recall \eqref{relacionhy1})
\[
S_1 = K^1_s(1) \bigl(-u_{j+1}+2u_{j}-u_{j-1}\bigr),\quad\text{and}\quad
S_2 = \sum_{|m|>1}K^1_s(m)(u_j-u_{j-m}).
\]
Next we show that $K^1_s(1)\to1$, while $S_2\to0$, as $s\to1^-$, which would give the conclusion. By~\eqref{eq:frKernelOned} with $h=1$ we have
\[
\lim_{s\to 1^-}K^1_s(1) = \lim_{s\to 1^-} \frac{4^s\Gamma(1-s)\Gamma(1/2+s)}{\sqrt{\pi}|\Gamma(-s)|\Gamma(2+s)}
= \lim_{s\to 1^-}\frac{4^s\Gamma(1/2+s)}{\sqrt{\pi}\Gamma(2+s)}
= \frac{4\Gamma(3/2)}{\sqrt{\pi}\Gamma(3)}=1.
\]
On the other hand, by~\eqref{eq:frKernelOned} with $h=1$, $S_2$ is bounded by
\begin{align*}
2\|u\|_{\ell^\infty_h}\sum_{|m|>1}K^1_s(m)\le 2\|u\|_{\ell^\infty}
\frac{4^s\Gamma(\frac{1}{2}+s)}{\pi^{1/2}|\Gamma(-s)|}
\sum_{|m|>1}\frac{\Gamma(|m|-s)}{\Gamma(|m|+1+s)},
\end{align*}
which goes to zero as $s\to1^-$.\qed

\begin{rem}
In \cite[formula~(5)]{CLRV} the following equivalent expression for the kernel of $(-\Delta_1)^s$ is presented:
for $m\neq0$,
$$
K_s^1(m)=\frac{(-1)^{m+1}\Gamma(2s+1)}{\Gamma(1+s+m)\Gamma(1+s-m)}.
$$
Indeed, apply the duplication formula
and Euler's reflection formula for the Gamma function to~\eqref{eq:frKernelOned}.
\end{rem}

\section{Proof of Theorem~\ref{lem:kernelFractionalIntegral}}
\label{sec:dicreteFractional}

A function $f:\Z_h\to\R$ is in $\ell^p_h$, $1\leq p<\infty$ if
\begin{equation}
\label{Lph}
\|f\|_{\ell^p_h}=\bigg(h\sum_{j\in\Z}|f_j|^p\bigg)^{1/p}<\infty,
\end{equation}
while $f\in\ell^\infty_h$ if
\begin{equation}
\label{Linftyh}
\|f\|_{\ell^\infty_h}=\sup_{hj\in\Z_h}|f_j|<\infty.
\end{equation}
Obviously $\ell^p_h\subset\ell^q_h$ if $1\leq p\leq q\leq\infty$, 
with $\|f\|_{\ell^q_h}\leq h^{1/q-1/p}\|f\|_{\ell^p_h}$.
The discrete H\"older's inequality takes the form
\begin{equation}
\label{Holder}
\|fg\|_{\ell^1_h}\leq\|f\|_{\ell^p_h}\|g\|_{\ell^{p'}_h},
\quad\text{for } 1\leq p\leq\infty,\ \tfrac{1}{p}+\tfrac{1}{p'}=1.
\end{equation}
When $h=1$ we write $\ell^p=\ell^p_1=\ell^p(\Z)$.

\noindent$(a)$. Observe that if $f\in\ell_{-s}$ then the semigroup $e^{t\Delta_h}f_j$
is well defined, for each $hj\in\Z_h$. This follows from an analogous computation
to that of \eqref{eq:cuentita}.
By writing down the semidiscrete heat kernel into
\eqref{eq:fractionalintegralsemigrupo} and using Fubini's theorem
(which will be fully justified once we prove \eqref{eq:kernelFrIntegralOned}
and \eqref{eq:growthFractional0}) we see that \eqref{formulapuntualpotenciasnegativas}
follows with
$$
K^h_{-s}(m) = \frac{1}{\Gamma(s)} \int_0^{\infty}G(m,\tfrac{t}{h^2})\frac{dt}{t^{1-s}}
= \frac{1}{h^{-2s}\Gamma(s)}\int_0^{\infty}G(m,r)\frac{dr}{r^{1-s}}.
$$
To get \eqref{eq:kernelFrIntegralOned}, we just use the expression in~\eqref{eq:semidiscreteheatkernel}
and formula~\eqref{eq:rusos} in the integral above.

\noindent$(b)$. The estimates in~\eqref{eq:growthFractional0}
and~\eqref{eq:growthFractional} follow from \eqref{eq:kernelFrIntegralOned} and Lemma~\ref{eq:lowerBound}.

\noindent$(c)$. We recall that Stein and Wainger
showed in \cite[Proposition~(a)]{Stein-Wainger} that the operator
$$
I_\lambda g_j=\sum_{m\in\Z,m\neq 0}\frac{g_{j-m}}{|m|^\lambda},\quad j\in\Z,~0<\lambda<1,
$$
acting on functions $g:\Z\to\R$, is bounded from $\ell^r$ into $\ell^l$, whenever
$1/l\leq 1/r-1+\lambda$ and $1<r<l<\infty$.
By using Minkoswki's inequality,
the estimate for the kernel $K^h_{-s}(m)$ in \eqref{eq:growthFractional0},
the boundedness of the operator $I_\lambda$
above with $0<\lambda=1-2s<1$ and $r=p$, $l=q$ as in our statement,
the inclusion $\ell^p\subset\ell^q$ and \eqref{Lph}, we get
\begin{align*}
\|(-\Delta_h)^{-s}f\|_{\ell^q_h} &=\bigg(\sum_{j\in \Z}\bigg|\sum_{m\in \Z}K^h_{-s}(m)f_{j-m}\bigg|^q\bigg)^{1/q} \\
&=\bigg(\sum_{j\in\Z}\bigg|\sum_{m\in\Z,m\neq0}K^h_{-s}(m)f_{j-m}+K_{-s}^h(0)f_{j}\bigg|^q\bigg)^{1/q}\\
&\le\bigg(\sum_{j\in \Z}\bigg|\sum_{m\in \Z,m\neq0}K^h_{-s}(m)f_{j-m}\bigg|^q\Bigg)^{1/q}
+C_sh^{2s}\bigg(\sum_{j\in\Z}|f_j|^q\bigg)^{1/q}\\
&\le C_sh^{2s}\bigg(\sum_{j\in \Z}\bigg|\sum_{m\in \Z,m\neq 0}\frac{f_{j-m}}{|m|^{1-2s}}\bigg|^q\bigg)^{1/q}
+C_sh^{2s}\bigg(\sum_{j\in\Z}|f_j|^p\bigg)^{1/p}\\
&\le C_{p,q,s}h^{2s}\bigg(\sum_{j\in\Z}|f_j|^p\bigg)^{1/p}+C_sh^{2s-1/p}\|f\|_{\ell^p_h}
\le C_{p,q,s}h^{2s-1/p}\|f\|_{\ell^p_h}.
\end{align*}
Multiply both sides by $h^{1/q}$ and recall \eqref{Lph} to reach~\eqref{HLS}.
\qed

\section{Proof of Theorems~\ref{thm:Holder}~and~\ref{thm:consistencia1d}}
\label{Section:dosteoremas}

For the reader's convenience, we recall the definition of H\"older spaces on the real line.

\begin{defn}[Continuous H\"older spaces]
\label{def:contH}
Given $k\in \N_0$ and $0<\alpha\leq1$, we say that
a continuous function $U:\R\to\R$ belongs to the H\"older space $C^{k,\alpha}$ if $U\in C^k$ and
\[
[U]_{C^{k,\alpha}} \equiv[U^{(k)}]_{C^{0,\alpha}} := \sup_{\substack{x,y\in \R\\x\neq y}}
\frac{|U^{(k)}(x)-U^{(k)}(y)|}{|x-y|^{\alpha}}<\infty,
\]
where $U^{(k)}$ denotes the $k$-th derivative of~$U$. The norm in the spaces $C^{k,\alpha}$ is given by
\[
\|U\|_{C^{k,\alpha}} := \sum_{l=0}^k\|U^{(l)}\|_{L^\infty} +[U^{(k)}]_{C^{0,\alpha}}.
\]
\end{defn}

Next we define the discrete H\"older spaces on the mesh $\Z_h$.
For $u:\Z_h\to\R$ we consider the first order difference operators
\begin{equation}
\label{eq:definitiondiscretederivatives}
D_+u_{j}=\frac{1}{h}(u_{j+1}-u_{j}),\quad\text{and}\quad
D_-u_{j}=\frac{1}{h}(u_{j}-u_{j-1}).
\end{equation}
For $\gamma, \eta\in \N_0$,
let $D_{+,-}^{\gamma,\eta}u_j:=D_{+}^{\gamma}
D_{-}^{\eta}u_j$, where
$D_{\pm}^{k}u$ means that we apply $k$-times the operator $D_{\pm}$ to $u$,
with $D_{\pm}^0u=u$.

\begin{defn}[Discrete H\"older spaces]
\label{def:discreteH}
Let $k\in \N_0$ and $0<\alpha\le1$. A
function $u:\Z_h\to\R$ belongs to the discrete H\"older space $C_h^{k,\alpha}$ if
there is a constant $C>0$ such that
\[
[u]_{C^{k,\alpha}_h}\equiv[D_{+,-}^{\gamma,\eta}u]_{C_h^{0,\alpha}} :=
\sum_{\gamma,\eta:\gamma+\eta=k}\sup_{\substack{hm,hj\in\Z_h\\m\neq j}}\frac{|D_{+,-}^{\gamma,\eta}u_{j}-
D_{+,-}^{\gamma,\eta} u_{m}|}{|hj-hm|^{\alpha}}\leq C<\infty.
\]
\end{defn}

\begin{rem}
It is obvious that if $u:\Z_h\to\R$ is bounded then it belongs to $C^{0,\alpha}_h$, for any $0<\alpha\leq1$.
In this case a discrete H\"older norm can be given by
$\|u\|_{C_h^{0,\alpha}}:=\|u\|_{\ell^\infty_h}+[u]_{C_h^{0,\alpha}}$.
\end{rem}

\subsection{Proof of Theorem~\ref{thm:Holder}}~

\noindent$(i)$. It suffices to prove the case $k=0$,
since $D_{\pm}$ commutes with $(-\Delta_h)^s$.
Let $hk,hj\in\Z_h$. By recalling \eqref{relacionhy1}, we can write
\begin{equation}
\label{diferenciaS}
|(-\Delta_h)^su_{k}-(-\Delta_h)^su_{j}|= \frac{1}{h^{2s}}|S_1+S_2|,
\end{equation}
where
\begin{equation}
\label{S1}
S_1 := \sum_{1\leq|m|\le|k-j|}\big(u_{k}-u_{k+m}-u_{j}+u_{j+m}\big)K^1_s(m),
\end{equation}
and $S_2$ is the rest of the sum over $|m|>|k-j|$.
By the kernel estimate~\eqref{eq:frKernelEst},
\[
S_1 \le C_{s}2[u]_{C_h^{0,\alpha}}h^{\alpha} \sum_{1\leq|m|\le|k-j|}
\frac{|m|^{\alpha}}{|m|^{1+2s}} \le C_{s}
[u]_{C_h^{0,\alpha}}h^{\alpha}|k-j|^{\alpha-2s}.
\]
For $S_2$ we use that $|u_{k}-u_{j}|\le [u]_{C_h^{0,\alpha}}h^{\alpha}|k-j|^{\alpha}$ and
\eqref{eq:frKernelEst} again to get
\[
S_2\le C_{s}[u]_{C_h^{\alpha}}h^{\alpha}|k-j|^{\alpha} \sum_{|m|>|k-j|}|m|^{-1-2s}\le C_{s}
[u]_{C_h^{0,\alpha}}h^{\alpha}|k-j|^{\alpha-2s}.
\]
We conclude by pasting together both estimates into \eqref{diferenciaS}.

\noindent$(ii)$. As in $(i)$, it is enough to consider just the case $k=0$.
We are going to use~\eqref{diferenciaS}.
Without loss of generality, let $m\in \N$. We split the sum in \eqref{diferenciaS}--\eqref{S1}
by taking the terms $u_{k}-u_{k+m}$ and $u_j-u_{j+m}$ separately.
The following computation works for both terms, so we do it only for the first one. It is verified that
\begin{equation}
\label{eq:discreteTVM}
u_{k+m}-u_{k}=h\sum_{\gamma=0}^{m-1}D_+u_{k+\gamma}.
\end{equation}
Therefore,
\begin{equation}
\label{vaso}
u_{k}-u_{k+m}=\Big(hmD_{+}u_{k}-
h\sum_{\gamma=0}^{m-1}D_+u_{k+\gamma}\Big)
-hmD_{+}u_{k}.
\end{equation}
On one hand, by taking into account that the kernel $K^1_s(m)$ is even, we get
\begin{equation}
\label{eq:nulo}
\sum_{1\leq|m|\le|k-j|}(hm D_{+}u_{k})K^1_s(m)
=hD_{+}u_{k}\sum_{1\leq|m|\le|k-j|}mK^1_s(m)=0.
\end{equation}
On the other hand, since $u\in C_h^{1,\alpha}$,
the first term in the right hand side of \eqref{vaso} can be bounded by
\begin{equation}
\label{eq:truco}
\begin{aligned}
h\sum_{\gamma=0}^{m-1}\big|D_{+}u_{k}-D_{+}u_{k+\gamma}\big| &\le h^{1+\alpha}[u]_{C^{1,\alpha}_h}
\sum_{\gamma=0}^{m-1}|\gamma|^{\alpha} \\
&\le h^{1+\alpha}[u]_{C^{1,\alpha}_h}|m|^{\alpha}|m|
=[u]_{C^{1,\alpha}_h}(h|m|)^{1+\alpha}.
\end{aligned}
\end{equation}
Using \eqref{eq:nulo} and~\eqref{eq:truco} (and their analogous for $u_{j}-u_{j+m}$)
in~\eqref{S1}, we conclude that
\[
|S_1|\le C_{s}[u]_{C^{1,\alpha}_h}h^{1+\alpha}\sum_{1\leq|m|\le|k-j|}
\frac{|m|^{1+\alpha}}{|m|^{1+2s}}
\le C_{s}[u]_{C^{1,\alpha}_h}(h|k-j|)^{1+\alpha-2s}.
\]
Now we deal with $S_2$. By~\eqref{eq:discreteTVM},
\begin{align*}
\big|(u_{k}-u_{j})-(u_{k+m}-u_{j+m})\big|&=\big|(u_{(k-j)+j}-u_{j})-(u_{(k-j)+(j+m)}-u_{j+m})\big|\\
&\le h\sum_{\gamma=0}^{k-j-1}|D_{+}u_{j+\gamma}-D_{+}u_{j+m+\gamma}|\le [u]_{C^{1,\alpha}_h}h^{1+\alpha}|m
|^{\alpha}|k-j|.
\end{align*}
Hence,
\[
|S_2| \le C[u]_{C^{1,\alpha}_h}h^{1+\alpha}|k-j| \sum_{|m|>|k-j|}|m|^{\alpha}
K^1_s(m)\le C[u]_{C^{1,\alpha}_h}h^{1+\alpha}|k-j|^{1+\alpha-2s}.
\]
\qed

\subsection{Proof of Theorem~\ref{thm:consistencia1d}}
\label{sec:ProofConsistency}

We need a preliminary lemma.

\begin{lem}
\label{lem:compa}
Let $0<s<1$ and let $A_s>0$ be as in \eqref{eq:constante1d}. Given $j\in \Z$, we have
\begin{equation}
\label{eq:lem1}
\bigg|A_s\int_{|y-h(j+m)|<h/2}\frac{dy}{|hj-y|^{1+2s}}-K^h_s(m)\bigg|\le \frac{C_s}{h^{2s}|m|^{2+2s}}, \quad \text{for all }\, m\in \Z\setminus\{0\},
\end{equation}
\begin{equation}
\label{eq:lem2}
\int_{|y-h(j+m)|<h/2}\frac{dy}{|hj-y|^{1+2s}}\le \frac{C_s}{h^{2s}|m|^{1+2s}},\quad \text{for all }\, m\in \Z\setminus\{0\},
\end{equation}
and
\begin{equation}
\label{eq:lem3}
\sum_{m\in \Z} \int_{|y-h(j+m)|<h/2} \frac{hj-y}{|hj-y|^{1+2s}}\,dy = 0.
\end{equation}
\end{lem}

\begin{proof}
Let $m\in \Z\setminus\{0\}$.
The change of variable $hj-y=hz$ and \eqref{relacionhy1} produce
\begin{multline*}
\bigg|\frac{A_s}{h^{2s}} \int_{|z-m|<1/2} \frac{dz}{|z|^{1+2s}}-K^h_s(m)\bigg| \\
\le \bigg|\frac{A_s}{h^{2s}} \int_{|z-m|<1/2} \bigg(\frac{1}{|z|^{1+2s}}-\frac{1}{|m|^{1+2s}}\bigg)\,dz\bigg|
+ h^{-2s}\bigg|\frac{A_s}{|m|^{1+2s}}-K^1_s(m)\bigg|.
\end{multline*}
By using the mean value theorem,
\[
\bigg|\int_{|z-m|<1/2} \bigg(\frac{1}{|z|^{1+2s}}-\frac{1}{|m|^{1+2s}}\bigg)\,dz\bigg|
\le C_s\bigg|\int_{|z-m|<1/2}\frac{dz}{|m|^{2+2s}}\bigg| = \frac{C_s}{|m|^{2+2s}},
\]
while by Lemma~\ref{eq:lowerBound},
\[
\bigg|\frac{A_s}{|m|^{1+2s}}-K^1_s(m)\bigg|\le \frac{C_s}{|m|^{2+2s}}.
\]
Thus \eqref{eq:lem1} follows. For \eqref{eq:lem2}, it is easy to see that
\[
\int_{|y-(h(j+m))|<h/2}\frac{dy}{|hj-y|^{1+2s}}
\le C_s\int_{|y-(h(j+m))|<h/2}\frac{dy}{|hm|^{1+2s}}
= \frac{C_s}{h^{2s}|m|^{1+2s}}.
\]
Finally, let us prove~\eqref{eq:lem3}. By symmetry, we have
\[
\int_{|y-hj|<h/2}\frac{(hj-y)}{|hj-y|^{1+2s}}\,dy = 0.
\]
Moreover, by changing variables $hj-y=z$, we get
\begin{align*}
\sum_{\substack{m\in \Z\\ m\neq 0}}
\int_{|z-hm|< h/2} \frac{z}{|z|^{1+2s}}\,dz
&= \sum_{\substack{\ell\in \Z\\ \ell\neq 0}}
\int_{|z+h\ell|< h/2} \frac{z}{|z|^{1+2s}}\,dz
= \sum_{\substack{\ell\in \Z\\ \ell\neq 0}}
\int_{|r-h\ell|< h/2} \frac{-r}{|r|^{1+2s}}\,dr,
\end{align*}
and the conclusion readily follows.
\end{proof}

Now we present the proof of Theorem~\ref{thm:consistencia1d}.

\noindent$(i)$. We write, for each $j\in \Z$,
\begin{align*}
\big(r_h\big((-\Delta)^sU\big)\big)_j
&= A_{s}\sum_{m\in \Z}\int_{|y-h(j+m)|< h/2} \frac{U(hj)-U(y)}{|hj-y|^{1+2s}}\,dy\\
&= A_{s}\bigg[\int_{|y-hj|< h/2} \frac{U(hj)-U(y)}{|hj-y|^{1+2s}}\,dy\\
&\qquad\quad+\sum_{\substack{m\in \Z\\ m\neq 0}}\int_{|y-h(j+m)|< h/2} \frac{U(h(j+m))-U(y)}{|hj-y|^{1+2s}}\,dy\\
&\qquad\quad + \sum_{\substack{m\in \Z\\ m\neq 0}} \big(U(hj)-U(h(j+m))\big)
\int_{|y-h(j+m)|< h/2} \frac{dy}{|hj-y|^{1+2s}}\bigg]\\
& =: A_s(S_0+S_1+S_2).
\end{align*}
We readily notice that
\[
|S_0|\le [U]_{C^{0,\alpha}} \int_{|hj-y|\le h/2}|hj-y|^{\alpha-2s-1}\,dy
\le C_s[U]_{C^{0,\alpha}}h^{\alpha-2s}.
\]
By using that $U\in C^{0,\alpha}$ and~\eqref{eq:lem2}, we have
\begin{align*}
|S_1| &\le C[U]_{C^{0,\alpha}} \sum_{\substack{m\in \Z\\ m\neq 0}}
\int_{|y-h(j+m)|< h/2} \frac{h^{\alpha}\,dy}{|hj-y|^{1+2s}} \\
&\le C_s[U]_{C^{0,\alpha}}h^{\alpha}\sum_{\substack{m\in \Z\\ m\neq 0}} \frac{1}{h^{2s}|m|^{1+2s}}
= C_s[U]_{C^{0,\alpha}}h^{\alpha-2s}.
\end{align*}
Now we compare $A_sS_2$ with $(-\Delta_h)^s(r_hU)_j$.
Since $U\in C^{0,\alpha}$, by Lemma~\ref{lem:compa} we can see that
\begin{align*}
\bigg|A_s\sum_{\substack{m\in \Z\\ m\neq 0}}
&\big(U(hj)-U(h(j+m))\big)\int_{|y-h(j+m)|< h/2}
\frac{dy}{|hj-y|^{1+2s}}-(-\Delta_h)^s(r_hU)_j\bigg|\\
&\le \sum_{\substack{m\in \Z\\m\neq 0}}
\big|U(hj)-U(h(j+m))\big| \bigg|A_s\int_{|y-h(j+m)|< h/2} \frac{dy}{|hj-y|^{1+2s}}-K^h_s(m) \bigg| \\
&\le C_s[U]_{C^{0,\alpha}}\sum_{\substack{m\in \Z\\m\neq 0}} \frac{|hm|^{\alpha}}{h^{2s}|m|^{2+2s}}
\le C_s[U]_{C^{0,\alpha}}h^{\alpha-2s}.
\end{align*}

\noindent$(ii)$. Observe that
$\frac{d}{dx}$ and $D_+$ commute with $(-\Delta)^s$ and $(-\Delta_h)^s$, respectively. Then
\begin{multline*}
\big\|D_+(-\Delta_h)^s (r_h U)- r_h\big(\tfrac{d}{dx}(-\Delta)^s U\big) \big\|_{\ell^\infty} \\
\le  \big\|(-\Delta_h)^s D_+(r_h U) - (-\Delta_h)^s \big(r_h \tfrac{d}{dx}U\big) \big\|_{\ell^\infty}
+ \big\| (-\Delta_h)^s \big(r_h \tfrac{d}{dx}U\big) - r_h\big(\tfrac{d}{dx}(-\Delta)^s U\big) \big\|_{\ell^\infty}.
\end{multline*}
For the second term we just apply $(i)$. As for the first one, by using the mean value theorem,
\begin{align*}
\big| (-\Delta_h)^s D_+&(r_h U)_j- (-\Delta_h)^s  \big(r_h \tfrac{d}{dx}U \big)_j \big| \\
&= \bigg|\sum_{\substack{m\in \Z\\m\neq 0}} K^h_s(m) \bigg[ \bigg( \frac{U(h(j+1))- U(hj)}{h} - U'(hj) \bigg) \\
&\qquad\qquad-\bigg(\frac{U(h(j+m+1))- U(h(j+m)))}{h}- U'(h(j+m))\bigg)\bigg]\bigg| \\
&= \bigg|\sum_{\substack{m\in \Z\\m\neq 0}} K^h_s(m) \bigg[\big(U'(\xi_j)-U'(hj)\big) -
\big( U'(\xi_{j+m})- U'(h(j+m)) \big)\bigg]\bigg| \\
&\le C[U]_{C^{1,\alpha}}\sum_{\substack{m\in \Z\\m\neq 0}} K^h_s(m)  h^{\alpha} \le C[U]_{C^{1,\alpha}}h^{\alpha-2s},
\end{align*}
where $\xi_j$ is an intermediate point between $hj$ and $h(j+1)$, and analogously~$\xi_{j+m}$.

\noindent$(iii)$. By taking into account~\eqref{eq:lem3}, we can write
\begin{align*}
& r_h\big((-\Delta)^sU\big)_j
= A_{s}\sum_{m\in \Z}\int_{|y-h(j+m)|< h/2} \frac{U(hj)-U(y)-U'(hj)(hj-y)}{|hj-y|^{1+2s}}\,dy \\
& \quad= A_{s}\bigg[\int_{|y-hj|< h/2} \frac{U(hj)-U(y)-U'(hj)(hj-y)}{|hj-y|^{1+2s}}\,dy \\
&\qquad\quad+\sum_{\substack{m\in \Z\\ m\neq 0}}
\int_{|y-h(j+m)|< h/2} \frac{U(h(j+m))-U(y)-U'(hj)(h(j+m)-y)}{|hj-y|^{1+2s}}\,dy \\
& \qquad\quad+ \sum_{\substack{m\in \Z\\ m\neq 0}}
\big(U(hj)-U(h(j+m))-U'(hj)(hj-h(j+m))\big)
\int_{|y-h(j+m)|< h/2} \frac{dy}{|hj-y|^{1+2s}}\bigg] \\
& \quad=: A_s(T_0+T_1+T_2).
\end{align*}
For $T_0$, we use the mean value theorem and the hypothesis on~$U$.
Indeed, if $|y-hj|< h/2$ and $\xi_j(y)$ is an intermediate point between $hj$ and $y$, we have
\begin{align*}
\frac{\big(U'(\xi_j(y))-U'(hj)\big)(hj-y)}{|hj-y|^{1+2s}}
&\le [U]_{C^{1,\alpha}}\frac{|\xi_j(y)-hj|^{\alpha}|hj-y|}{|hj-y|^{1+2s}}\\
&\le [U]_{C^{1,\alpha}}|hj-y|^{\alpha-2s}.
\end{align*}
Then, as a consequence,
$$
|T_0|\le [U]_{C^{1,\alpha}}\int_{|y-hj|< h/2} |hj-y|^{\alpha-2s}\,dy
\le  C_s[U]_{C^{1,\alpha}}h^{\alpha-2s+1},
$$
whenever $2s<1+\alpha$.
By the hypotheses and~\eqref{eq:lem2},
\[
|T_1|\le C_s[U]_{C^{1,\alpha}} \sum_{\substack{m\in \Z\\ m\neq 0}}
\frac{|hm|^{\alpha}h}{h^{2s}|m|^{1+2s}} = C_s[U]_{C^{1,\alpha}}h^{1+\alpha-2s}.
\]
We compare $A_sT_2$ with $(-\Delta_h)^s(r_hU)_j$. Since $K^h_s(m)$ is even in $m$, we can write
\[
(-\Delta_h)^s(r_hU)_j=\sum_{\substack{m\in \Z\\ m\neq 0}} \big(U(hj)-U(h(j+m))-U'(hj)(hj-h(j+m))\big)K^h_s(m).
\]
Then \eqref{eq:lem1} and the regularity of $U$ give the result.

\noindent$(iv)$. The proof in this case follows as in $(ii)$ by iteration $l$ times.
\qed

\section{Proof of Theorem \ref{RegularityfractionalIntegral}}
\label{Section:JoseLuis}

We shall need two lemmas.

\begin{lem}
\label{derivative}
Let $0<s<1/2$ and  $\displaystyle H_s(r):=\int_0^\infty e^{-(r+s)v} (1-e^{-v})^{-2s} dv$, for $r >0$.
For any $k\geq0$ there exists a constant $C_{k,s} >0$ such that
$$
\Big|\frac{d^k}{dr^k}H_s(r)\Big|\le\frac{C_{k,s}}{(r+s)^{k+1-2s}},\quad\text{for all } r>0.
$$
\end{lem}

\begin{proof}
We have
\begin{align*}
\Big|\frac{d^k}{dr^k}H_s(r)\Big|&=\Big|(-1)^k\int_0^{\infty}e^{-(r+s)v}(1-e^{-v})^{-2s}v^k\,dv\Big|
\\&=\int_0^1+\int_1^{\infty} e^{-(r+s)v}(1-e^{-v})^{-2s}v^k\,dv=:I_1+I_2.
\end{align*}
On one hand, since $1-e^{-v}\ge(1-e^{-1})v=Cv$ for $v\in (0,1)$,
\begin{align*}
I_1\le C_s\int_0^1e^{-(r+s)v}v^{k-2s}\,dv&=C_s\int_0^{r+s}e^{-t}\frac{t^{k-2s}}{(r+s)^{k+1-2s}}\,dt\\
&\le \frac{C_s}{(r+s)^{k+1-2s}}\int_0^{\infty}e^{-t}t^{k-2s}\,dt\\
&=C_s\frac{\Gamma(k+1-2s)}{(r+s)^{k+1-2s}}.
\end{align*}
On the other hand,
$$
I_2\le \int_1^{\infty}e^{-(r+s)v}v^k\,dv = (r+s)^{-(k+1)}
\int_{r+s}^{\infty}e^{-t}t^k\,dt \le \frac{\Gamma(k+1)}{(r+s)^{k+1}}.
$$
By collecting both estimates, we conclude that
\begin{equation*}
\Big|\frac{d^k}{dr^k}H_s(r)\Big|
\le C_s\frac{\Gamma(k+1-2s)}{(r+s)^{k+1-2s}}+\frac{\Gamma(k+1)}{(r+s)^{k+1}}\\
\le \frac{C_{k,s}}{(r+s)^{k+1-2s}},
\end{equation*}
because $(r+s)^{-(k+1)}\le C_s (r+s)^{-(k+1-2s)}$.
\end{proof}
Recall the identity for the quotient of Gamma functions in \cite[Section 7 (15)]{Tricomi-Erdelyi}:
\begin{equation}
\label{eq:tricomi}
\frac{\Gamma(z+\alpha)}{\Gamma(z+\beta)}=\frac{1}{\Gamma(\beta-\alpha)}\int_0^{\infty}e^{-(z+\alpha)v}(1-e^{-v})^{\beta-\alpha-1}\,dv,
\end{equation}
valid for $\Re(\beta-\alpha)>0$, $\Re(z+\alpha)>0$.
It follows from \eqref{eq:kernelFrIntegralOned} and \eqref{eq:tricomi} with $z=|m|$,
$\alpha=s\in(0,1/2)$ and $\beta=1-s$, that
\begin{equation}
\label{conexion}
K^h_{-s}(m)=C_sh^{2s}H_s(|m|),\quad m\in\Z,
\end{equation}
for some constant $C_s>0$, where $H_s$ is the function we defined in Lemma~\ref{derivative}.

\begin{lem}
\label{cero}
Let $0<s<1/2$ and $j,k\in \Z$. Then
$$
\sum_{m\in\Z}\big(K^h_{-s}(m-j)-K^h_{-s}(m-k)\big)=0.
$$
\end{lem}

\begin{proof}
Clearly it is enough to prove that for every positive integer $j$ we have
$$
\sum_{m\in\Z}\big(K^1_{-s}(m-j)-K^1_{-s}(m)\big)=0,
$$
where
\begin{equation}
\label{verde}
K_{-s}^1(m)=h^{-2s}K_s^h(m),\quad m\in\Z.
\end{equation}
Observe that by Lemma \ref{derivative} the series above is absolutely convergent.
On the other hand, by the symmetry of the kernel $K^1_{-s}$ we have
$$
\sum_{m<0}\big(K^1_{-s}(m-j)-K^1_{-s}(m)\big)=\sum_{m>0}\big(K^1_{-s}(m+j)-K^1_{-s}(m)\big),
$$
and
$$
\sum_{m>j}\big(K^1_{-s}(m-j)-K^1_{-s}(m)\big)=\sum_{m>0}\big(K^1_{-s}(m)-K^1_{-s}(m+j)\big).
$$
Finally,
\begin{align*}
\sum_{0\le m\le j}\big(K^1_{-s}(m-j)-K^1_{-s}(m)\big) &=K^1_{-s}(-j)-K^1_{-s}(0)+
K^1_{-s}(1-j)-K^1_{-s}(1) \\
&\quad+\cdots+K^1_{-s}(-1)-K^1_{-s}(j-1)+K^1_{-s}(0)-K^1_{-s}(j)\\
&=0.
\end{align*}
Pasting together these computations we get the claim.
\end{proof}

We are ready to begin with the proof of Theorem \ref{RegularityfractionalIntegral}.
We shall prove the result only for the case $h=1$.
The general case $h>0$ follows by using the relation \eqref{verde}.

\noindent $(i).$ As the discrete derivatives commute with $(-\Delta_1)^{-s}$ for $0<s<1/2$,
it suffices to prove the case $k=0$.
Moreover, it is enough to show that for any positive $j$ we have
$$
|(-\Delta_1)^{-s}f_j-(-\Delta_1)^{-s}f_0|\le C[f]_{C_1^{0,\alpha}}j^{\alpha+2s}.
$$
By using Lemma \ref{cero} we can write
\begin{align*}
|(-\Delta_1)^{-s} f_j-  (-\Delta_1)^{-s} f_0|&=
\bigg|\sum_{m\in\Z}\big(K^1_{-s}(m-j)-K^1_{-s}(m)\big)(f_m-f_0)\bigg| \\
&\le \Big(\sum_{0<|m|\le 2j}+ \sum_{|m|>2j}\Big)|K^1_{-s}(m-j)-K^1_{-s}(m)||f_m-f_0| \\
&=:  S_1+S_2.
\end{align*}
By using the estimate in \eqref{eq:growthFractional0} we get
\begin{align*}
S_1 & \le  C[ f]_{C_1^{0,\alpha}}\bigg(K^1_{-s}(0) j^\alpha+
\sum_{0<|m|\le 2j,m\neq j}  \frac{|m|^\alpha}{|m-j|^{1-2s}} +
\sum_{0<|m|\le2j}  \frac{|m|^\alpha}{|m|^{1-2s}}\bigg) \\
&\le C[f]_{C^{0,\alpha}_1}j^\alpha\bigg(1+\sum_{0<|m-j|\le3j}\frac1{|m-j|^{1-2s}} +
 \sum_{0<|m|\le 2j}  \frac1{|m|^{1-2s}}\bigg) \\
& \le  C [f]_{C^{0,\alpha}_1}j^{\alpha+2s}.
\end{align*}
On the other hand, \eqref{conexion}, the mean value theorem and Lemma \ref{derivative} with $k=1$ allow us to estimate
 \begin{align*}
S_2 &\le C [f]_{C_1^{0,\alpha}}\sum_{|m|> 2j}|H_s(|m-j|)- H_s(|m|)||m|^\alpha \\
&\le C [ f]_{C_1^{0,\alpha}}j \sum_{|m|> j}  \frac{|m|^\alpha}{|m|^{2-2s}} \le
 C [f]_{C_1^{0,\alpha}} j^{\alpha+2s}.
\end{align*}

\noindent$(ii).$ Again, it is enough to prove only the case $k=0$. By Lemma~\ref{cero},
\begin{align*}
D_+&((-\Delta_1)^{-s}f_j)-D_+((-\Delta_1)^{-s}f_0) \\
&=\big((-\Delta_1)^{-s}f_{j+1}-(-\Delta_1)^{-s}f_j\big)-\big((-\Delta_1)^{-s}f_1-(-\Delta_1)^{-s}f_0\big) \\
&=\sum_{|m|>0}\Big[\big(K^1_{-s}(m-(j+1))-K^1_{-s}(m-j)\big)-\big(K^1_{-s}(m-1)-K^1_{-s}(m)\big)\Big](f_m-f_0).
\end{align*}
Proceeding as in $(i)$, we decompose into the sums
$T_1=\sum_{0<|m|\le2j}$ and $T_2=\sum_{|m|>2j}$. To estimate $T_1$,
we use the estimates
$$
|K^1_{-s}(m-1) - K^1_{-s}(m)|\le\frac{C_s}{|m|^{2-2s}}, \qquad m\neq 0,
$$
and
$$
|K_{-s}(m-(j+1))- K_{-s}(m-j)| \le\frac{C_s}{|m-j|^{2-2s}}, \qquad m\neq j.
$$ 
They are deduced from \eqref{conexion}, the mean value theorem and Lemma \ref{derivative} with $k=1$ for $m\neq 1$ and $m\neq j+1$, respectively. The particular cases $m=1$ and $m=j+1$ are trivial.

Then, by observing that $2s <1$, we can proceed as in $(i)$,
arriving at $T_1\le C [ f]_{C_1^{0,\alpha}}j^{\alpha+2s-1}.$
Regarding the term $T_2$, we write, up to a multiplicative constant depending on~$s$ (see~\eqref{conexion}),
\begin{multline*}
\big(K^1_{-s}(m-(j+1))-K^1_{-s}(m-j)\big)-\big(K^1_{-s}(m-1)-K^1_{-s}(m)\big) \\
=H_s(|m-(j+1)|)-H_s(|m-j|) -(H_s(|m-1|)-H_s(|m|)).
\end{multline*}
By a repeated application of the mean value theorem and Lemma \ref{derivative} with $k=2$ we then get
$$
\big|\big(K^1_{-s}(m-(j+1))-K^1_{-s}(m-j)\big)-\big(K^1_{-s}(m-1)-K^1_{-s}(m)\big)\big|
\le C_s\frac{j}{|m|^{3-2s}}.
$$
Hence
$$
T_2\leq C[ f]_{C_1^{0,\alpha}}j\sum_{|m|>j}\frac{|m|^\alpha}{|m|^{3-2s}}
\le C[f]_{C_1^{0,\alpha}}j^{\alpha+2s-1}.
$$

\noindent $(iii).$ The proof of $(i)$ can be adapted to this case, details are left to the interested reader.

\qed

\section{The Dirichlet problem for the fractional discrete Laplacian}
\label{Section:Dirichlet}

Throughout this section we fix a finite interval $B^h\subset\Z_h$.
The aim of this section is to show the following existence and uniqueness result.

\begin{thm}[Discrete Dirichlet problem]\label{thm:discreteDirichletproblem}
Let $0<s<{1/2}$ and $f:B^h\to\R$.
Then there exists a unique solution $u:\Z_h\to\R$ to the nonlocal discrete Dirichlet problem
\begin{equation}
\label{laecuacion}
\begin{cases}
(-\Delta_h)^su=f, &\text{in } B^h, \\
u=0, &\text{in } \Z_h\setminus B^h.
\end{cases}
\end{equation}
\end{thm}

Before presenting the proof we need some preliminaries.

We first observe that if $u:\Z_h\to\R$ is a bounded function then
$(-\Delta_h)^su$ is well defined and bounded, with
\begin{equation*}
h^{2s}\|(-\Delta_h)^su\|_{\ell^\infty_h}\leq C_s\|u\|_{\ell^\infty_h}.
\end{equation*}
Indeed, for any $hj\in\Z_h$, by \eqref{eq:puntualdiscreta} and \eqref{eq:frKernelEst},
$$
|(-\Delta_h)^su_j|\leq \frac{C_s}{h^{2s}}\sum_{m\neq j}\frac{2\|u\|_{\ell^\infty_h}}{|j-m|^{1+2s}}
=\frac{C_s}{h^{2s}}\|u\|_{\ell^\infty_h}\sum_{m\neq0}\frac{1}{|m|^{1+2s}}.
$$
In particular, $(-\Delta_h)^su$ is a well defined bounded function
whenever $u\in\ell^p_h$, for any $1\leq p\leq\infty$.
We also observe that, for $0<s<1$,
\begin{equation}
\label{acotacionL2h}
\text{if } u\in\ell^2_h \text{ then } (-\Delta_h)^su\in\ell^2_h,
\text{ with } \|(-\Delta_h)^su\|_{\ell^2_h}\leq
\frac{4^s}{h^{2s}}\|u\|_{\ell^2_h}.
\end{equation}
This follows, for example, by using the Fourier transform, which we now introduce.
Let $\T_h=\R/(2\pi\Z_h)=\R/(2\pi h\Z)$, the one dimensional torus of length $2\pi h$, which we identify
with the interval $[-h\pi,h\pi)$. We denote $L^2_h=L^2(\T_h)$ with inner product
$$
\langle U,V\rangle_{L^2_h} = \int_{-h\pi}^{h\pi}U(\theta)\overline{V(\theta)}\,d\theta.
$$
Then the set of exponentials
$\big\{(2\pi h)^{-1/2}e^{ij\theta/h}:j\in\Z,\theta\in\T_h\big\}$,
where $i$ denotes the imaginary unit,
forms an orthonormal basis of $L^2_h$.
For an integrable function $U:\T_h\to\R$, its Fourier series is given by
$$
S[U](\theta)=\frac{1}{(2\pi h)^{1/2}}\sum_{j\in\Z}\widehat{U}(j)e^{ij\theta/h},
$$
where
$$
\widehat{U}(j)=\frac{1}{(2\pi h)^{1/2}}\int_{-h\pi}^{h\pi}U(\theta)e^{-ij\theta/h}\,d\theta,\quad j\in\Z.
$$
Given $u:\Z_h\to\R$, its Fourier transform is a function defined on $[-h\pi,h\pi)$
whose Fourier coefficients are given by the sequence $\{u_j\}_{j\in\Z}$. In other words,
if $u\in\ell^1_h$ then we define
$$
\mathcal{F}_{\Z_h}u(\theta)=\sum_{j\in\Z}u_je^{ij\theta/h},\quad\theta\in[-h\pi,h\pi).
$$
Then the operator $u\mapsto\mathcal{F}_{\Z_h}u$ extends as an isometry from $\ell^2_h$
into $L^2_h$, with inverse given by
$$
\mathcal{F}^{-1}_{\Z_h}U(j)=\widehat{U}(j).
$$
Let us then prove \eqref{acotacionL2h}. We can easily check that
if $u\in\ell^2_h$ then
$$
\mathcal{F}_{\Z_h}(-\Delta_hu)(\theta) =
\bigg[\frac{4}{h^2}\sin^2\Big(\frac{\theta}{2h}\Big)\bigg]\mathcal{F}_{\Z_h}u(\theta).
$$
It is a simple exercise to verify that our semigroup definition \eqref{definition}
coincides with the Fourier transform definition
\begin{equation}
\label{Fouriertransformdefinition}
\mathcal{F}_{\Z_h}\big[(-\Delta_h)^su\big](\theta)
= \Big[\frac{4}{h^2}\sin^2\Big(\frac{\theta}{2h}\Big)\Big]^s
\mathcal{F}_{\Z_h}u(\theta),
\end{equation}
for $0<s<1$. Then \eqref{acotacionL2h} follows by noticing that the Fourier
multiplier
$$
m_s(\theta)=\bigg[\frac{4}{h^2}\sin^2\Big(\frac{\theta}{2h}\Big)\bigg]^s,\quad\theta\in\T_h,
$$
is a bounded function on $[-h\pi,h\pi)$, with
$$
\|m_s\|_{L^\infty(\T_h)}=\frac{4^s}{h^{2s}}.
$$

\begin{lem}
\label{lemmaquefalta}
Let $u,v\in\ell^2_h$. Then, for any $0<s<1$,
\begin{align*}
\langle(-\Delta_h)^{s}u,v\rangle_{\ell^2_h} 
&= \langle(-\Delta_h)^{s/2}u,(-\Delta_h)^{s/2}v\rangle_{\ell^2_h} \\
&= \frac{h}{2}\sum_{j\in\Z}\sum_{m\in\Z,m\neq j}(u_j-u_m)(v_j-v_m)K^h_s(j-m).
\end{align*}
\end{lem}

\begin{proof}
In view of \eqref{acotacionL2h} we can use Plancherel's identity
and the Fourier transform characterization \eqref{Fouriertransformdefinition} to write
\begin{equation}
\label{equation1}
\begin{aligned}
\langle(-\Delta_h)^{s/2}u,(-\Delta_h)^{s/2}v\rangle_{\ell^2_h}
&=\langle u,(-\Delta_h)^{s}v\rangle_{\ell^2_h} \\
&=h\sum_{j\in\Z}u_j(-\Delta_h)^{s}v_j \\
&=h\sum_{j\in\Z}\sum_{m\in\Z,m\neq j}u_j\big(v_j-v_m\big)K_s^h(j-m).
\end{aligned}
\end{equation}
By interchanging the roles of $j$ and $m$ above and using
Fubini's Theorem and the symmetry $K_s^h(m-j)=K_s^h(j-m)$,
we can also write
\begin{equation}
\label{equation2}
\begin{aligned}
\langle(-\Delta_h)^{s/2}u,(-\Delta_h)^{s/2}v\rangle_{\ell^2_h}
&=h\sum_{m\in\Z}\sum_{j\in\Z,j\neq m}u_m\big(v_m-v_j\big)K_s^h(m-j) \\
&=-h\sum_{j\in\Z}\sum_{m\in\Z,m\neq j}u_m\big(v_j-v_m\big)K_s^h(j-m).
\end{aligned}
\end{equation}
After adding \eqref{equation1} and \eqref{equation2} we get the conclusion.
\end{proof}

\begin{rem}
It is clear from \eqref{equation1} and the Cauchy--Schwarz inequality that
if $u\in\ell^2_h$ then the following interpolation inequality holds:
$$
\|(-\Delta_h)^{s/2}u\|_{\ell^2_h}\leq\|u\|_{\ell^2_h}\|(-\Delta_h)^su\|_{\ell^2_h}.
$$
\end{rem}

The following important result will be proved in Subsection \ref{subsection:Sobolev}.

\begin{thm}[Fractional discrete Sobolev and Poincar\'e inequalities]
\label{thm:SobolevandPoincare}
Let $0<s<1/2$.
There is a constant $C_s>0$ depending only on $s$ such that the fractional discrete Sobolev inequality
$$
\|u\|_{\ell^{2/(1-2s)}_h}\leq C_s\|(-\Delta_h)^{s/2}u\|_{\ell^2_h} =
C_s\bigg(\frac{h}{2}\sum_{j\in\Z}\sum_{m\in\Z,m\neq j}|u_j-u_m|^2K^h_s(j-m)\bigg)^{1/2},
$$
holds for any function $u:\Z_h\to\R$ with compact support $\supp(u)\subset\Z_h$. In particular,
we have the fractional discrete Poincar\'e inequality
\begin{equation}
\label{discretePoincareinequality}
\|u\|_{\ell^2_h}\leq C_sh^s\big(\#_h\supp(u)\big)^s\|(-\Delta_h)^{s/2}u\|_{\ell^2_h},
\end{equation}
where, for a set $E\subset\Z_h$, the notation $\#_hE$ means
the number of points in~$E$.
\end{thm}

\begin{lem}
Let $0<s<1/2$. If we endow the set of functions
$$
H_0^s(B^h):=\big\{u:\Z_h\to\R:u=0 \text{ on } \Z_h\setminus B^h\big\},
$$
with the inner product
$$
\langle u,v\rangle_{H_0^s(B^h)} = \langle(-\Delta_h)^{s/2}u,(-\Delta_h)^{s/2}v\rangle_{\ell^2_h},
$$
then $H_0^s(B^h)$ is a Hilbert space.
\end{lem}

\begin{proof}
Clearly $H_0^s(B^h)$ is a linear space and the bilinear form $\langle u,v\rangle_{H_0^s(B^h)}$ is symmetric,
with $\langle u,u\rangle_{H^s_0(B^h)}\geq0$ for all $u\in H_0^s(B^h)$.
Let us check that $\langle u,u\rangle_{H^s_0(B^h)}=0$ implies $u=0$. Indeed, we have
\[
0\leq\frac{h}{2}\sum_{j\in\Z}\sum_{m\in\Z,m\neq j}|u_j-u_m|^2K^h_s(j-m)=0.
\]
Since the kernel $K_s^h(j-m)$ is positive for $j\neq m$ (see \eqref{eq:frKernelOned}),
we get $u_j=u_m$ for all $j\neq m$.
As $u_m$ is zero for all $m$ outside $B^h$, we get $u=0$ everywhere.
Another way of proving that $u=0$ is by means of the fractional discrete Poincar\'e inequality
\eqref{discretePoincareinequality}.
To show completeness, suppose that $(u^k)_{k\ge0}$ is a Cauchy
sequence in~$H_0^s(B^h)$. Then, by the Poincar\'e inequality \eqref{discretePoincareinequality},
we see that $(u^k)_{k\ge0}$ is a Cauchy sequence in $\ell^2_h$
and so it has a pointwise limit $u\in\ell^2_h$. Observe that $u=0$ in $\Z_h\setminus B^h$
and so, in view of \eqref{acotacionL2h}, $u\in H_0^s(B^h)$. Moreover,
again by \eqref{acotacionL2h}, $u^k\to u$ in $H^s_0(B^h)$, as $k\to\infty$.
\end{proof}

\begin{proof}[Proof of Theorem \ref{thm:discreteDirichletproblem}]
We say that $u:\Z_h\to\R$ is a weak solution to \eqref{laecuacion} if $u\in H_0^s(B^h)$ and
$$
\langle u,v\rangle_{H^s_0(B^h)}=\langle(-\Delta_h)^{s/2}u,(-\Delta_h)^{s/2}v\rangle_{\ell^2_h}
= h\sum_{j\in\Z}f_jv_j=\langle f,v\rangle_{\ell^2_h},
$$
for all $v\in H_0^s(B^h)$.
Let us show that $v\mapsto\langle f,v\rangle_{\ell^2_h}$ is a bounded linear functional in~$H_0^s(B^h)$.
By H\"older's inequality \eqref{Holder} with $p=p'=2$ and the Poincar\'e inequality~\eqref{discretePoincareinequality},
\[
\bigg|h\sum_{j\in\Z}f_jv_j\bigg|\leq \|f\|_{\ell^2_h}\|v\|_{\ell^2_h}
\leq \|f\|_{\ell^2_h}C_sh^s\big(\#_h B^h\big)^s\|v\|_{H_0^s(B^h)}.
\]
Hence the Riesz representation theorem applies and shows that for any
given $f$ there is a unique weak solution $u\in H_0^s(B_h)$. The fact that the first equation
in \eqref{laecuacion} holds (that is, that $u$ is a classical solution)
follows because Lemma \ref{lemmaquefalta} shows that
$\langle(-\Delta_h)^su,v\rangle_{\ell^2_h}=\langle f,v\rangle_{\ell^2_h}$, for all $v:\Z_h\to\R$
such that $v=0$ in $\Z_h\setminus B^h$.
\end{proof}

\section{The discrete maximum principle}
\label{sec:Dirichlet}

To complete the proof of Theorem \ref{thm:Convergence-result} we need the following maximum principle.

\begin{thm}[Discrete maximum principle]
\label{thm:DiscretePoisson1d}
Let $0<s<1$. Fix an interval $B^h_R\subset\Z_h$, $R>0$.
Suppose that $f\in \ell_h^\infty(B^h_R)$ and $g\in\ell_h^\infty(\Z_h\setminus B^h_R)$.
If $u$ is a solution to
\begin{equation}
\label{eq:DDp}
\begin{cases}
(-\Delta_h)^su=f,&\text{in}~B^h_R, \\
u=g, & \text{in}~\Z_h\setminus B^h_R.
\end{cases}
\end{equation}
Then there is a universal constant $C>0$ independent of $s$, $R$ and $h$ such that
\[
\|u\|_{\ell_h^\infty(B^h_R)}\leq CR^{2s}
\|f\|_{\ell_h^\infty(B^h_R)}+\|g\|_{\ell_h^\infty(\Z_h\setminus B_R^h)}.
\]
In particular, uniqueness holds for the Dirichlet problem \eqref{eq:DDp}.
\end{thm}

For the proof we need a barrier which is constructed in
Lemma~\ref{lem:barrier} and the nonlocal maximum principle stated in Lemma~\ref{lem:maximum}
(see also \cite{Huang-Oberman}).

By the symmetry of the kernel $K_s^h(m)$ in $m\neq0$ we can always write
\begin{equation}
\label{eq:rewrite}
(-\Delta_h)^su_{j}= \frac{1}{2} \sum_{m\neq 0}
\big(2u_{j}-u_{j-m}-u_{j+m}\big)K^h_s(m).
\end{equation}

\begin{lem}
\label{lem:barrier}
Let $0<s<1$ and $R>0$. Define
$$
W(x)=\begin{cases}
4R^2-|x|^2, & \text{for}~|x|<R, \\
0, &\text{for}~|x|\geq R.
\end{cases}
$$
Then the function $w_j := (r_hW)_j = W(hj)$ satisfies
\[
\begin{cases}
(-\Delta_h)^sw\ge M R^{2-2s},& \text{ in } B_R^h, \\
w=0, & \text{ in } \Z_h\setminus B_R^h,
\end{cases}
\]
where $M>0$ is a constant independent of $R$, $s$ and $h$.
\end{lem}

\begin{proof}
First, for each $hm\in\Z_h$, it is not difficult to prove that
\[
2w_j-w_{j-m}-w_{j+m}\ge \min\{2R^2, 2|hm|^2\}, \quad \text{ for } hj \in B_R^h,
\]
see, for example, the proof of \cite[Lemma 5]{Huang-Oberman} for $R=1$.
With this and taking into account \eqref{eq:rewrite} we have, for $hj\in B_R^h$,
\[
(-\Delta_h)^sw_{j} \ge \frac{1}{2} \sum_{m\neq 0}
\min\{2R^2, 2|hm|^2\}K^h_s(m)
= R^2\sum_{|hm|\ge R}K^h_s(m)+h^2
\sum_{\substack {|hm|< R\\ m\neq 0}}|m|^2K^h_s(m).
\]
Now we use the explicit expression \eqref{eq:frKernelOned} and
Lemma~\ref{eq:lowerBound}$(b)$. Then, there exist constants $C_1$
and $C_2$ independent of $R$, $s$ and $h$ such that, on one hand,
\begin{align*}
R^2\sum_{|hm|\ge R}K^h_s(m)&\ge \frac{R^2}{h^{2s}}
\sum_{|m|\ge R/h} \frac{4^s\Gamma(1/2+s)}{2^{1+2s}\sqrt{\pi}|\Gamma(-s)|}
\frac{1}{|m|^{1+2s}} \\
&\ge C_1 \frac{\Gamma(1/2+s)}{2\sqrt{\pi}|\Gamma(-s)|} \frac{R^2}{h^{2s}}
\int_{|x|\ge R/h} \frac{1}{|x|^{1+2s}}\,dx \\
&= C_1\frac{\Gamma(1/2+s)}{2s\sqrt{\pi}|\Gamma(-s)|}R^{2-2s},
\end{align*}
and, on the other hand,
\begin{align*}
h^2\sum_{\substack {|hm|< R\\ m\neq 0}}|m|^2K^h_s(m)
&\ge \frac{h^2}{h^{2s}}
\sum_{\substack {|m|\le R/h\\ m\neq 0}}
\frac{4^s\Gamma(1/2+s)}{2^{1+2s}\sqrt{\pi}|\Gamma(-s)|} \frac{|m|^2}{|m|^{1+2s}} \\
&\ge C_2\frac{\Gamma(1/2+s)}{2\sqrt{\pi}|\Gamma(-s)|}h^{2-2s}
\int_{|x|\le R/h}\frac{|x|^2}{|x|^{1+2s}}\,dx\\
&=C_2\frac{\Gamma(1/2+s)}{(2-2s)\sqrt{\pi}|\Gamma(-s)|}R^{2-2s}.
\end{align*}
Altogether,
\[
(-\Delta_h)^sw_{j}\ge \min\{C_1,C_2\}\frac{\Gamma(1/2+s)}{\sqrt{\pi}|\Gamma(-s)|}
\bigg[\frac{1}{2s}+\frac{1}{2-2s}\bigg]R^{2-2s}>MR^{2-2s},
\]
where $M>0$ is a constant independent of $R$,$s$ and $h$,
because $\frac{\Gamma(1/2+s)}{\sqrt{\pi}|\Gamma(-s)|}
\big(\frac{1}{2s}+\frac{1}{2-2s}\big)>\frac14$.
\end{proof}

\begin{lem}
\label{lem:maximum}
Let $v:\Z_h \to \R$ be a subsolution to
$(-\Delta_h)^sv\le 0$ in an interval $B^h\subset\Z_h$. Then
\[
\max_{B^h}v\le \sup_{\Z_h\setminus B^h}v.
\]
Similarly, if $v:\Z_h\to\R$ is a supersolution to $(-\Delta_h)^sv\ge 0$ in $B^h\subset\Z_h$ then
\[
\min_{B^h}v\ge \inf_{\Z_h\setminus B^h}v.
\]
\end{lem}

\begin{proof}
By considering $-v$ in place of $v$, it is enough to prove the result for subsolutions. We argue by contradiction.
Suppose that the maximum of $v$ in $B^h$, which is attained at a point $hj_0\in B^h$, is strictly larger than
$\sup_{\Z_h\setminus B^h}v$. Then there exists $hm_0\neq hj_0$ such that $v_{j_0}>v_{m_0}$. Hence,
by hypothesis and since $v_{j_0}-v_m\geq0$ for all $m\neq j_0$, we get
$$
0\geq(-\Delta_h)^sv_{j_0}=\sum_{m\neq j_0}(v_{j_0}-v_m)K^h_s(j_0-m)
\geq(v_{j_0}-v_{m_0})K^h_s(j_0-m_0)>0,
$$
which is a contradiction.
Then the maximum of $v$ in $B^h$ cannot be strictly larger than $\sup_{\Z_h\setminus B^h}v$.
\end{proof}

\begin{proof}[Proof of Theorem~\ref{thm:DiscretePoisson1d}]
Set
$$
v=u-M^{-1}R^{2s-2} \|(-\Delta_h)^su\|_{\ell_h^{\infty}(B_R^h)}w,
$$
where $w$ and $M$ are as in Lemma~\ref{lem:barrier}. Then, for any $hj\in B_R^h$,
\[
(-\Delta_h)^sv_{j}=(-\Delta_h)^su_{j}
-M^{-1}R^{2s-2}\|(-\Delta_h)^su\|_{\ell_h^{\infty}(B_R^h)}(-\Delta_h)^sw_j\le 0.
\]
Thus, by the maximum principle in Lemma~\ref{lem:maximum},
$\max_{B_R^h}v\le \sup_{\Z_h\setminus B_R^h}v$. On the other hand, since $w=0$ on $\Z_h\setminus B_R^h$, we have
\[
\sup_{\Z_h\setminus B_R^h}v= \sup_{\Z_h\setminus B_R^h}u.
\]
Thus, as $0\leq w\le 4R^2$ on $B_R^h$,
\begin{align*}
\max_{B_R^h}u &\le \sup_{\Z_h\setminus B_R^h}v
+ M^{-1}R^{2s-2}\|(-\Delta_h)^su\|_{\ell_h^{\infty}(B_R^h)} \max_{B_R^h}w \\
&\le \|u\|_{\ell_h^{\infty}(\Z_h\setminus B_R^h)}
+ 4M^{-1}R^{2s}\|(-\Delta_h)^su\|_{\ell_h^{\infty}(B_R^h)}\\
&=\|g\|_{\ell_h^\infty(\Z_h\setminus B_R^h)}+4M^{-1} R^{2s}
\|f\|_{\ell_h^\infty(B_R^h)}.
\end{align*}
Similarly, it can be proved that
$$
\min_{B_R^h}u \ge -\|u\|_{\ell_h^{\infty}(\Z_h\setminus B_R^h)}
- 4M^{-1}R^{2s}\|(-\Delta_h)^su\|_{\ell_h^{\infty} (B_R^h)}.
$$
The proof is complete.
\end{proof}

\section{Proof of Theorem~\ref{thm:Convergence-result}}
\label{sec:proofmain}

Let $v=r_hU-u$. Then
\begin{equation*}
\begin{cases}
(-\Delta_h)^sv=(-\Delta_h)^sr_hU-f, & \text{ in } B_R^h, \\
v=r_hU, & \text{ in } \Z_h\setminus B_R^h.
\end{cases}
\end{equation*}
By Theorem~\ref{thm:DiscretePoisson1d},
\[
\|r_hU-u\|_{\ell_h^\infty(B_R^h)}\le CR^{2s}\|(-\Delta_h)^{s}(r_hU)-f
\|_{\ell^\infty_h(B_R^h)} + \|r_hU\|_{\ell^\infty_h(\Z_h\setminus B_R^h)}.
\]
Since $f=r_h((-\Delta)^sU)$ and $U\in C^{0,\alpha+2s}$,
Theorem~\ref{thm:consistencia1d}$(i)$ implies that the first term above
is bounded by $CR^{2s}[U]_{C^{0,\alpha+2s}}h^{\alpha}$,
where $C$ is independent of $R$ and $h$.
For the second term, we clearly have
$\|r_hU\|_{\ell^\infty(\Z_h\setminus B_R^h)}\le \|U\|_{L^\infty(\R\setminus B_R)}$.
From the results in Theorem \ref{thm:continuousPoissonproblem}, we have
$$
[U]_{C^{0,\alpha+2s}}\leq C_{\alpha,s,R_0}\|F\|_{C^{0,\alpha}},
$$
see \eqref{eq:laregularidaddeU}.
Moreover, for any $x\in \R\setminus B_R$ with $R>2R_0$, we have (see \eqref{eq:laregularidaddeU})
\begin{equation}
\label{laestimaciondeloqueresta}
|U(x)|\le A_{-s}\int_{B_{R_0}}\frac{|F(y)|}{|x-y|^{1-2s}}\,dy
\le A_{-s}\frac{\|F\|_{L^\infty}}{R_0^{1-2s}}\le C\frac{\|F\|_{L^\infty}}{R^{1-2s}},
\end{equation}
where $C>0$ depends only on $s$.
Hence, with our choice of $R$,
\[
\|r_hU-u\|_{\ell^\infty(B_R^h)}
\le C_{\alpha,s,R_0}\|F\|_{C^{0,\alpha}}R^{2s}\big(h^{\alpha}+R^{-1}\big)
\leq C_{\alpha,s,R_0}\|F\|_{C^{0,\alpha}}R^{2s}h^{\alpha},
\]
where $C_{\alpha,s,R_0}>0$ is independent of $R$ and $h$.\qed

\section{Technical lemmas, Bessel functions, the continuous Poisson problem
and the fractional discrete Sobolev and Poincar\'e inequalities}
\label{sec:technical}

\subsection{Some technical lemmas}
\label{subsec:technical-lemmas}

Lemmas in this subsection are needed in the proof of Theorem~\ref{thm:consistencia1d}.
They are also useful to get estimates for the kernels of the fractional discrete Laplacian
in Theorem~\ref{thm:basicProperties} and for the
fractional integral kernel in Theorem~\ref{lem:kernelFractionalIntegral}.

\begin{lem}
\label{lem:Mario}
Let $\lambda>0$. Let $a,b$ be real numbers such that $0\le a<b<\infty$. Then
\[
\min\{\lambda,1\}\le\frac{b^{\lambda}-a^{\lambda}}{b^{\lambda-1}(b-a)}\le \max\{\lambda,1\}.
\]
\end{lem}

\begin{proof}
Let us first suppose that $\lambda\ge1$. Then, as $0\le a<b<\infty$, we have
\[
0\le a/b<1\Rightarrow 0\le(a/b)^{\lambda}
\le a/b<1\Rightarrow \frac{b^{\lambda}-a^{\lambda}}{b^{\lambda-1}(b-a)}=\frac{1-(a/b)^{\lambda}}{1-a/b}\ge1.
\]
On the other hand, by applying the mean value theorem to the function $t\mapsto t^{\lambda}$, we get
\[
\frac{b^{\lambda}-a^{\lambda}}{b^{\lambda-1}(b-a)}=\frac{1-(a/b)^{\lambda}}{1-a/b}=\lambda x^{\lambda-1}\le \lambda,
\]
for certain $x\in (a/b,1)$. In the case $0<\lambda<1$, the proof is analogous.
\end{proof}

\begin{lem}
\label{eq:lowerBound}
Let $0<s<1$, $t\in \R$, and $m\in \Z$, $m\neq 0$.
\begin{enumerate}[$(a)$]
\item We have
\begin{equation}
\label{eq:lowerLap}
\bigg|\frac{\Gamma(|m|-s)}{\Gamma(|m|+1+s)}-\frac{1}{|m|^{1+2s}}\bigg|\le \frac{C_{s}}{|m|^{2+2s}}.
\end{equation}
In the case when $0<s<1/2$, we have
\begin{equation}
\label{eq:lowerFrac}
\bigg|\frac{\Gamma(|m|+s)}{\Gamma(|m|+1-s)}-\frac{1}{|m|^{1-2s}}\bigg|\le \frac{C_{s}}{|m|^{2-2s}}.
\end{equation}
The constants $C_s>0$ above depend only on $s$.
\item We have the lower bounds
\[
\frac{\Gamma(|m|-s)}{\Gamma(|m|+1+s)}\ge\frac{1}{(2|m|)^{1+2s}},
\]
and, for $0<s<1/2$,
\[
\frac{\Gamma(|m|+s)}{\Gamma(|m|+1-s)}\ge\frac{1}{(2|m|)^{1-2s}}.
\]

\end{enumerate}
\end{lem}

\begin{proof}
Without loss of generality, $m>0$. We begin with the proof of~\eqref{eq:lowerLap} in $(a)$.  We write
\begin{align*}
\bigg|\frac{\Gamma(m-s)}{\Gamma(m+1+s)}-\frac{1}{m^{1+2s}}\bigg|
&\le\bigg|\frac{\Gamma(m-s)}{\Gamma(m+1+s)}-\frac{1}{(m-s)^{1+2s}}\bigg| +\bigg|\frac{1}{(m-s)^{1+2s}}-\frac{1}{m^{1+2s}}\bigg|.
\end{align*}
The second term can be easily estimated, just by applying Lemma~\ref{lem:Mario}
with $\lambda=1+2s$, $a=\frac{1}{m}$ and $b=\frac{1}{m-s}$, namely,
\[
\bigg|\frac{1}{(m-s)^{1+2s}}-\frac{1}{m^{1+2s}}\bigg|
\sim\frac{1}{(m-s)^{2s}}\bigg(\frac{1}{m-s}-\frac{1}{m}\bigg)\sim\frac{C_{s}}{m^{2+2s}},
\]
where the symbol $\sim$ means that constants depend only on~$s$.
Now we study the first term.
By recalling \eqref{eq:tricomi}, for $k\in \N$, we have
$$
\frac{\Gamma(k-s)}{\Gamma(k+n+s)}
= \frac{1}{\Gamma(n+2s)}\int_0^{\infty}e^{-(k-s)v}(1-e^{-v})^{n+2s-1}\,dv.
$$
With this,
\begin{align*}
\Gamma(1+2s)\bigg|\frac{\Gamma(m-s)}{\Gamma(m+1+s)}-\frac{1}{(m-s)^{1+2s}}\bigg|
&\le \int_0^{\infty}e^{-(m-s)v}\big|v^{2s}-(1-e^{-v})^{2s}\big|\,dv\\
&= \int_0^{\infty}e^{-(m-s)v}v^{2s}\bigg|1-\Big(\frac{1-e^{-v}}{v}\Big)^{2s}\bigg|\,dv\\
&\sim \int_0^{\infty}e^{-(m-s)v}v^{2s}\bigg|1-\frac{1-e^{-v}}{v}\bigg|\,dv\\\
&\le \frac12\int_0^{\infty}e^{-(m-s)v}v^{2s+1}\,dv\sim\frac{\Gamma(1+2s)}{2}\frac{1}{m^{2+2s}},
\end{align*}
where we applied Lemma~\ref{lem:Mario},
and in the last inequality we used that $\frac{v^2}{2}>v-1+e^{-v}$ for $v\in (0,\infty)$.
The proof of \eqref{eq:lowerFrac} is analogous, with the restriction $0<s<1/2$ coming from Lemma~\ref{lem:Mario}.

The proof of the first estimate in $(b)$
follows from \eqref{eq:tricomi} and an application of the Mean Value Theorem, namely,
\begin{align*}
\frac{1}{\Gamma(1+2s)}\int_0^{\infty}e^{-(m-s)v}(1-e^{-v})^{2s}\,dv&\ge \frac{1}{\Gamma(1+2s)}\int_0^{\infty}e^{-(m-s)v}e^{-2sv}v^{2s}\,dv\\
&=\frac{1}{(m+s)^{1+2s}}\ge\frac{1}{2^{1+2s}m^{1+2s}}.
\end{align*}
In a similar way we can get the second bound in $(b)$ after choosing
$z=m$, $\alpha=s\in(0,1/2)$ and $\beta=1-s$ in \eqref{eq:tricomi}.
\end{proof}

\subsection{Properties of Bessel functions $I_k$}
\label{sec:preliminaries}

We collect in this subsection some properties of modified Bessel functions.
Let $I_k$ be the modified Bessel function of the first kind and order $k\in \Z$, defined as
\begin{equation}
\label{eq:Ik}
I_k(t) = i^{-k} J_k(it) = \sum_{m=0}^{\infty} \frac{1}{m!\,\Gamma(m+k+1)} \left(\frac{t}{2}\right)^{2m+k}.
\end{equation}
Since $k$ is an integer and $1/\Gamma(n)$ is taken to be equal zero if $n=0,-1,-2,\ldots$, the function $I_k$ is defined in the whole real line.
It is verified that
\begin{equation}
\label{eq:negPos}
I_{-k}(t) = I_k(t),
\end{equation}
for each $k \in \Z$. Besides, from~\eqref{eq:Ik} it is clear that
$I_0(0) = 1$ and $I_k(0) = 0$ for $k \ne 0$.
Also,
\begin{equation}
\label{eq:Ik>0}
I_k(t) \ge 0
\end{equation}
for every $k \in \Z$ and $t\ge0$, and
\begin{equation}
\label{eq:sumIk}
\sum_{k \in \Z} e^{-2t} I_k(2t) = 1.
\end{equation}
On the other hand, there exist constants $C,c>0$ such that
$$
c t^k\le I_k(t)\le C t^k, \quad \text{ as } t\to0^+.
$$
In fact,
\begin{equation}
\label{eq:asymptotics-zero-ctes}
I_k(t)\sim \bigg(\frac{t}{2}\bigg)^k\frac{1}{\Gamma(k+1)}, \quad \text{ for a fixed } 
k\neq -1,-2,-3,\ldots~\text{and}~t\to0^+,
\end{equation}
see \cite{OlMax}.
It is well known (see \cite{Lebedev}) that
\begin{equation}
\label{eq:asymptotics-infinite}
I_k(t)=C e^t t^{-1/2}+ R_k(t),
\end{equation}
where
\[
|R_k(t)|\le C_k e^tt^{-3/2}, \quad \text{ as } t\to\infty.
\]
We also have (see \cite{OlMax}) that, as $\nu\to \infty$,
\begin{equation}
\label{eq:asymptotics-order-large}
I_{\nu}(z)\sim \frac{1}{\sqrt{2\pi \nu}}\bigg(\frac{e z}{2 \nu}\bigg)^{\nu}\sim \frac{z^\nu}{2^{\nu} \nu!}.
\end{equation}
For the following formula see \cite[p.~305]{Prudnikov2}. For $\Re c>0$, $-\Re \nu<\Re \alpha<1/2$,
\begin{equation}
\label{eq:rusos}
\int_0^{\infty} e^{-ct} I_{\nu}(ct) t^{\alpha-1}\,dt
= \frac{(2c)^{-\alpha}}{\sqrt{\pi}} \frac{\Gamma(1/2-\alpha)\Gamma(\alpha+\nu)
}{\Gamma(\nu+1-\alpha)}.
\end{equation}

\subsection{The fractional discrete Sobolev and Poincar\'e inequalities}
\label{subsection:Sobolev}

The aim of this subsection is to prove Theorem \ref{thm:SobolevandPoincare}.
The idea is to follow the clever elementary proof of the fractional Sobolev inequality
presented by O. Savin and E. Valdinoci in \cite{Savin-Valdinoci-1},
which makes use of their Sobolev embedding for sets proved in \cite{Savin-Valdinoci-2}.
We sketch the main steps.

Along this subsection we let
$$
0<s<1/2,\quad\text{so that}\quad 2/(1-2s)>2.
$$
We start by working on the integers $\Z$, then we will see how to get
the result for the mesh $\Z_h$, for any $h>0$. If $R>0$ we denote the discrete interval
$$
I_R:=\{n\in\Z:-R<n<R\}.
$$
In this way, if $R$ is an integer then
the measure of the interval above is
\begin{equation}
\label{medida}
\#I_R=2R-1,
\end{equation}
where we denote by $\#E$ the number of integers points in the set $E\subset\Z$
(counting measure).
As in \cite{Savin-Valdinoci-1}, the following discrete Sobolev embedding for sets
(which in fact is valid for any $s>0$) is crucial in the proof.

\begin{lem}[{Discrete analogue of \cite[Lemma~A.1]{Savin-Valdinoci-2}}]\label{lem:1}
Let $k\in\Z$ be fixed. Let $E\subset\Z$ be a nonempty finite set. There exists a constant $C_s>0$
depending only on $s$ such that
$$
\sum_{l\notin E}\frac{1}{|k-l|^{1+2s}}\geq C_s(\#E)^{-2s}.
$$
\end{lem}

\begin{proof}
We assume that $k\in E$, otherwise the conclusion is trivially true
as the left hand side of the inequality is infinite.
By replacing $E$ by $E-k$ we can suppose that $k=0\in E$.
Let $R=\#E$. Then $R$ is a positive integer. If we show that
\begin{equation}
\label{eq:toshow}
\sum_{l\notin E}\frac{1}{|l|^{1+2s}}\geq
\sum_{l\notin I_R}\frac{1}{|l|^{1+2s}},
\end{equation}
then the conclusion follows. Indeed, we can bound
$$
\sum_{l\notin I_R}\frac{1}{|l|^{1+2s}}
\geq \int_R^\infty\frac{1}{x^{1+2s}}\,dx=C_s(\#E)^{-2s}.
$$
For \eqref{eq:toshow}, we first observe that we can use \eqref{medida} to get
\begin{align*}
\#\big(E^c\cap I_R\big) &= \#I_R-\#\big(E\cap I_R\big) \\
&= (2R-1)-\#\big(E\cap I_R\big) \\
&\geq\#E-\#\big(E\cap I_R\big)=\#\big(E\cap I_R^c\big).
\end{align*}
We estimate now by using this last inequality as follows:
\begin{align*}
\sum_{l\notin E}\frac{1}{|l|^{1+2s}} &=
\sum_{l\notin E,l\in I_R}\frac{1}{|l|^{1+2s}}+\sum_{l\notin E,l\notin I_R}\frac{1}{|l|^{1+2s}} \\
&\geq \sum_{l\notin E,l\in I_R}\frac{1}{R^{1+2s}}+
\sum_{l\notin E,l\notin I_R}\frac{1}{|l|^{1+2s}} \\
&= \frac{\#\big(E^c\cap I_R\big)}{R^{1+2s}}+\sum_{l\notin E,l\notin I_R}\frac{1}{|l|^{1+2s}} \\
&\geq \frac{\#\big(E\cap I_R^c\big)}{R^{1+2s}}+
\sum_{l\notin E,l\notin I_R}\frac{1}{|l|^{1+2s}} \\
&\geq \sum_{l\in E,l\notin I_R}\frac{1}{|l|^{1+2s}}
+\sum_{l\notin E,l\notin I_R}\frac{1}{|l|^{1+2s}}= \sum_{l\notin I_R}\frac{1}{|l|^{1+2s}},
\end{align*}
and \eqref{eq:toshow} is proved.
\end{proof}

\begin{lem}[{See \cite[Lemma~5]{Savin-Valdinoci-1} with $T=2^2>1$ and $n=1$}]\label{lem:2}
Let $N\in\Z$ and let $a_j$ be a bounded, nonnegative, decreasing sequence with $a_j=0$
for all $j\geq N$. There is a constant $C_s>0$ depending only on $s$ such that
$$
\sum_{j\in\Z}2^{2k}a_j^{1-2s}\leq C_s\sum_{j\in\Z,a_j\neq0}2^{2j}a_{j+1}a_j^{-2s}.
$$
\end{lem}

\begin{lem}[{Discrete analogue of \cite[Lemma~6]{Savin-Valdinoci-1}}]
\label{lem:3}
Let $f:\Z\to\R$ be a sequence with compact support. Define
\begin{equation}
\label{eq:ak}
a_j:=\#\{k\in\Z:|f_k|>2^j\}.
\end{equation}
Then there is a constant $C_s>0$ depending only on $s$ such that
$$
\sum_{j\in\Z,a_j\neq0}2^{2j}a_{j+1}a_j^{-2s} \leq
C_s\sum_{j\in\Z}\sum_{m\in\Z,m\neq j}\frac{|f_j-f_m|^2}{|j-m|^{1+2s}}.
$$
\end{lem}

Using Lemmas \ref{lem:2} and \ref{lem:3} we can prove the following result.

\begin{thm}[Discrete analogue of the fractional
Sobolev inequality on $\Z$]
There is a constant $C_s>0$ depending only on $s$ such that
for any sequence $f:\Z\to\R$ with compact support,
\begin{equation}
\label{eq:discreteSobolev}
\bigg(\sum_{j\in\Z}|f_j|^{2/(1-2s)}\bigg)^{(1-2s)/2}\leq
C_s\bigg(\sum_{j\in\Z}\sum_{m\in\Z,m\neq j}\frac{|f_j-f_m|^2}{|j-m|^{1+2s}}\bigg)^{1/2}.
\end{equation}
\end{thm}

\begin{proof}
It is easy to see that the right hand side of \eqref{eq:discreteSobolev} is finite
(even for the more general case $f\in\ell^2$, see \eqref{acotacionL2h}, Lemma \ref{lemmaquefalta}
and \eqref{eq:frKernelEst}).
For any $j\in\Z$ we let
$A_j=\{k\in\Z:|f_k|>2^j\}$.
Notice that
\begin{equation}
\label{Ak}
A_j\supset A_{j+1},\quad\text{and}\quad\bigcup_{j\in\Z}A_j=\Z.
\end{equation}
Let $a_j=\#A_j$ as in \eqref{eq:ak}. We can write
\begin{align*}
\sum_{j\in\Z}|f_j|^{2/(1-2s)} &=\sum_{j\in\Z}\sum_{k\in A_j\setminus A_{j+1}}|f_k|^{2/(1-2s)} \\
&\leq \sum_{j\in\Z}(2^{j+1})^{2/(1-2s)}\#\big(A_j\setminus A_{j+1}\big) \\
&\leq \sum_{j\in\Z}2^{2(j+1)/(1-2s)}a_j.
\end{align*}
As $1-2s<1$, the function $\varphi(t)=t^{1-2s}$, $t\geq0$, is concave in $[0,\infty)$ and
satisfies $\varphi(0)=0$. Hence $\varphi$ is subadditive.
Using this and the estimate we just performed above, we get
\begin{equation}
\label{cuenta1}
\bigg(\sum_{j\in\Z}|f_j|^{2/(1-2s)}\bigg)^{1-2s}\leq 4\sum_{j\in\Z}2^{2j}a_j^{1-2s}.
\end{equation}
Next we verify that the sequence $a_j$ satisfies the hypotheses of Lemma \ref{lem:2}.
As $f$ has compact support, we have $A_j\subseteq\supp(f)$, for all $j\in\Z$. Then
$0\leq a_j\leq\#\supp(f)$,
so the sequence $a_j$ is uniformly bounded in $j\in\Z$ and each $a_j$ is nonnegative.
Using \eqref{Ak} it follows that $a_j$ is decreasing.
Finally, as $f$, being a compactly supported sequence of real numbers, is bounded,
there is an $N>0$ such that $|f_k|<2^N$ for all $k\in\Z$. Hence $A_j$ is empty for all $j\geq N$,
which gives that $a_j=0$ for all $j\geq N$. Thus we can apply Lemma \ref{lem:2} to the sequence $a_j$.
From \eqref{cuenta1}, by Lemma \ref{lem:2} and Lemma \ref{lem:3}, we clearly see that
\eqref{eq:discreteSobolev} follows.
\end{proof}

We are left to prove Lemma \ref{lem:3}, for which we follow \cite{Savin-Valdinoci-1}.

\begin{proof}[Proof of Lemma \ref{lem:3}]
As $||f_j|-|f_m||\leq|f_j-f_m|$, we can assume, by possibly replacing $f$ by $|f|$, that $f\geq0$.
For any $j\in\Z$, let us define
$$
D_j=A_j\setminus A_{j+1}=\{m\in\Z:2^j<f_m\leq2^{j+1}\},\quad\text{and}\quad d_j=\#D_j.
$$
As $f$ is bounded with compact support, both $a_j$ and $d_j$ become zero for $j$ large enough.
Define the convergent series
$$
S=\sum_{j\in\Z,a_{j-1}\neq0}2^{2j}a_{j-1}^{-2s}d_j.
$$
With this notation we have (see the computation for \cite[eq.~(32)]{Savin-Valdinoci-1}):
$$
\sum_{j\in\Z,a_{j-1}\neq0}\sum_{\ell\in\Z,\ell\geq j+1}2^{2j}a_{j-1}^{-2s}d_\ell\leq S.
$$
Let $j\in\Z$ and $k\in D_j$. Take any $m\leq j-2$ and any $l\in D_m$. Then
$$
|f_k-f_l|\geq 2^j-2^{m+1}\geq 2^j-2^{j-1}=2^{j-1},
$$
from which, by Lemma \ref{lem:1} and the facts that $\cup_{m\leq j-2}D_m=A_{j-1}^c$ (disjoint union) and
$a_{j-1}=\#A_{j-1}$, we deduce
\begin{align*}
\sum_{m\in\Z,m\leq j-2}\sum_{l\in D_m}\frac{|f_k-f_l|^2}{|k-l|^{1+2s}}
&\geq 2^{2(j-1)}\sum_{m\in\Z,m\leq j-2}\sum_{l\in D_m}
\frac{1}{|k-l|^{1+2s}} \\
&= 2^{2(j-1)}\sum_{l\notin A_{j-1}}\frac{1}{|k-l|^{1+2s}} \\
&\geq C_s2^{2j}a_{j-1}^{-2s}.
\end{align*}
Now we sum the inequality above among all $l\in D_m$ and use that
$$
d_j=a_j-\sum_{\ell\geq j+1}d_\ell
$$
to get that, for every $j\in\Z$,
\begin{align}
\sum_{m\in\Z,m\leq j-2}\sum_{l\in D_m}\sum_{k\in D_j}\frac{|f_k-f_l|^2}{|k-l|^{1+2s}}
&\geq C_02^{2j}a_{j-1}^{-2s}d_j \label{medio}\\
&= C_0\bigg[2^{2j}a_{j-1}^{-2s}a_j-\sum_{\ell\in\Z,\ell\geq j+1}2^{2j}a_{j-1}^{-2s}d_\ell\bigg].\nonumber
\end{align}
We sum for all $j\in\Z$ such that $a_{j-1}\neq0$ in inequality \eqref{medio} to get
$$
\sum_{j\in\Z,a_{j-1}\neq0} \sum_{m\in\Z,m\leq j-2}\sum_{l\in D_m}\sum_{k\in D_j}\frac{|f_k-f_l|^2}{|k-l|^{1+2s}}
\geq C_0\sum_{j\in\Z,a_{j-1}\neq0}2^{2j}a_{j-1}^{-2s}d_j=C_0S.
$$
Therefore, as in \cite[p.~2685]{Savin-Valdinoci-1}, we reach the analogue of
\cite[eq.~(36)]{Savin-Valdinoci-1}:
\begin{equation}
\label{cuenta}
2\sum_{j\in\Z,a_{j-1}\neq0}\sum_{m\in\Z,m\leq j-2}\sum_{l\in D_m}\sum_{k\in D_j}\frac{|f_k-f_l|^2}{|k-l|^{1+2s}}
\geq C_0\sum_{j\in\Z,a_{j-1}\neq0}2^{2j}a_{j-1}^{-2s}a_j.
\end{equation}
Finally, by symmetry, we can write
\begin{align*}
\sum_{j\in\Z}\sum_{m\in\Z,m\neq j}\frac{|f_j-f_m|^2}{|j-m|^{1+2s}} &= 2\sum_{j,m\in\Z,m<j}\frac{|f_j-f_m|^2}{|j-m|^{1+2s}} \\
&=2\sum_{j,m\in\Z,m<j}\sum_{k\in D_j}\sum_{l\in D_m}\frac{|f_k-f_l|^2}{|k-l|^{1+2s}} \\
&\geq 2\sum_{j\in\Z,a_{j-1}\neq0}\sum_{m\in\Z,m\leq j-2}\sum_{k\in D_j}\sum_{l\in D_m}\frac{|f_k-f_l|^2}{|k-l|^{1+2s}},
\end{align*}
and from \eqref{cuenta}, the conclusion of Lemma \ref{lem:3} follows with $C_s=1/C_0$.
\end{proof}

\begin{cor}[Discrete analogue of Sobolev inequality on $\Z_h$]\label{cor:h}
Let $u:\Z_h\to\R$ be a function with compact support. Then there is a constant $C_s>0$
depending only on $s$ such that
\begin{equation}
\label{eq:bla}
\|u\|_{\ell^{2/(1-2s)}_h}\leq C_s\bigg(h^2\sum_{j\in\Z}\sum_{m\in\Z,m\neq j}
\frac{|u_j-u_m|^2}{|hj-hm|^{1+2s}}\bigg)^{1/2}.
\end{equation}
\end{cor}

\begin{proof}
Given the function $u$ on $\Z_h$ we can define a new function (sequence) $f$
on $\Z_1=\Z$ through $f_j:=u(hj)=u_j$, for $j\in\Z$.
With this notation, for any $1\leq p\leq\infty$, $\|u\|_{\ell^p_h}=h^{1/p}\|f\|_{\ell^p}$.
Then $f$ is a sequence with compact support,
so we can apply \eqref{eq:discreteSobolev} to it and get
$$
\bigg(\sum_{j\in\Z}|u_j|^{2/(1-2s)}\bigg)^{(1-2s)/2} \leq
C_s\bigg(\sum_{j\in\Z}\sum_{m\in\Z,m\neq j} \frac{|u_j-u_m|^2}{|j-m|^{1+2s}}\bigg)^{1/2}.
$$
Now we multiply both sides by $h^{(1-2s)/2}$ and \eqref{eq:bla} follows.
\end{proof}

\begin{rem}
Observe that the factor $h^2$ appearing in the right hand side of \eqref{eq:bla} is the correct
one since that expression is nothing but the $\ell^2(\Z_h\times\Z_h)$ norm of the two-variables function
$v=v(hj,hm):\Z_h\times\Z_h\setminus\{(hj,hm):j=m\}\to\R$ given by
$$
v(hj,hm)=\frac{|u_j-u_m|^2}{|hj-hm|^{1+2s}}.
$$
\end{rem}

\begin{proof}[Proof of Theorem \ref{thm:SobolevandPoincare}]
By applying \eqref{eq:frKernelEst} and Corollary \ref{cor:h},
\begin{align*}
\frac{h}{2}\sum_{j\in\Z}\sum_{m\in\Z,m\neq j}|u_j-u_m|^2K^h_s(j-m)
&\geq C_sh^2\sum_{j\in\Z}\sum_{m\in\Z,m\neq j}\frac{|u_j-u_m|^2}{|hj-hm|^{1+2s}} \\
&\geq C_s\|u\|_{\ell^{2/(1-2s)}_h}^2.
\end{align*}
and the Sobolev inequality follows.
To prove the Poincar\'e inequality,
notice that, by H\"older's inequality \eqref{Holder} with $p=1/(1-2s)>1$ and $p'=1/(2s)$ ,
\begin{align*}
\|u\|_{\ell^2_h}^2 &=h\sum_{hj\in\supp(u)}|u_j|^2=\|\chi_{\supp(u)}\cdot u^2\|_{\ell^1_h} \\
&\leq \|\chi_{\supp(u)}\|_{\ell^{1/(2s)}_h}\|u^2\|_{\ell^{1/(1-2s)}_h} \\
&= h^{2s}\big(\#_h\supp(u)\big)^{2s}\bigg(h\sum_{j\in\Z}|u_j|^{2/(1-2s)}\bigg)^{1-2s} \\
&= h^{2s}\big(\#_h\supp(u)\big)^{2s}\|u\|_{\ell^{2/(1-2s)}_h}^2.
\end{align*}
Then we apply the Sobolev inequality.
\end{proof}

\subsection{The continuous Poisson problem}
\label{Subsection:Poissonproblem}

In this subsection we prove the following result, which we believe it belongs to the folklore.
We present here a more or less self contained proof.
From now on we denote the Fourier transform of $G\in L^1(\R)$ by
$$
\widehat{G}(\xi)=\frac{1}{(2\pi)^{1/2}}\int_{\R}G(x)e^{-i\xi x}\,dx,\quad\xi\in\R.
$$

\begin{thm}
\label{thm:continuousPoissonproblem}
Let $F$ be a function in $C^{0,\alpha}$, for some $0<\alpha<1$,
having compact support on~$\R$. Take $0<s<1$ such that $0<\alpha+2s<1$.
Then the function $U$ defined by
$$
U(x):=(-\Delta)^{-s}F(x)=A_{-s}\int_{\supp(F)}\frac{F(y)}{|x-y|^{1-2s}}\,dy,
$$
for $x\in\R$, where $\supp(F)$ denotes the support of $F$ and $A_{-s}>0$ is as in \eqref{constantfractionalintegral},
is the unique classical solution to the fractional Poisson problem \eqref{eq:PoissonProblem}
such that $|U(x)|\to0$ as $|x|\to\infty$. Moreover, $U\in C^{0,\alpha+2s}$ and there is a constant
$C>0$ depending only on $s$, $\alpha$ and the measure of $\supp(F)$, such that
\begin{equation}
\label{eq:laregularidaddeU}
\|U\|_{C^{0,\alpha+2s}}\leq C\|F\|_{C^{0,\alpha}}.
\end{equation}
\end{thm}

\begin{proof}
Without loss of generality we may assume
that $F=0$ outside an interval $(-R_0,R_0)$, for some $R_0>0$. Then we can write
$$
U(x)=A_{-s}\int_{-R_0}^{R_0}\frac{F(y)}{|x-y|^{1-2s}}\,dy.
$$
It is clear that $U$ is well defined because $|y|^{-1+2s}$ is a locally integrable function in $\R$
and $F$ is bounded.
Next we prove a series of properties about $U$ that will complete the proof.

\begin{enumerate}[$(1)$]

\item \textbf{$U$ is bounded on $\R$ and $|U(x)|\to0$ as $|x|\to\infty$.}
To see this we first do a computation. Let $r>0$.
It is easy to see that the positive function
$$
N_{s,r}(x):=\int_{-r}^r|x-y|^{-1+2s}\,dy,\qquad x\in\R,
$$
is H\"older continuous of order $0<2s<1$ on $\R$
and smooth in $\R\setminus\{-r,r\}$. In particular,
\begin{equation}
\label{N}
\frac{d}{dx}N_{s,r}(x)=(x+r)^{2s-1}-(r-x)^{2s-1},\qquad\text{for any } -r<x<r.
\end{equation}
Now, if $|y|<r$ and $|x|>2r$ then $|x-y|^{-1+2s}\leq
(|x|-r)^{-1+2s}$. Hence,
$$
|N_{s,r}(x)|\leq\frac{2r}{(|x|-r)^{1-2s}},\qquad\text{for any } |x|>2r,
$$
which shows that $|N_{s,r}(x)|\to 0$ as $|x|\to\infty$. Thus $N_{s,r}\in C^{0,2s}$.
For our claim (1) we just observe that
\begin{equation}
\label{acotaciondeUporF}
|U(x)|\leq A_{-s}\|F\|_{L^\infty}N_{s,R_0}(x),\qquad\text{for any } x\in\R.
\end{equation}

\item \textbf{Let $\varphi\in C^\infty(\R)$ such that $(1+|x|^{1+2s})D^k\varphi\in L^\infty(\R)$,
for all $k\geq0$ (that is, $\varphi\in\mathcal{S}_s$, see the notation in \cite[p.~73]{Silvestre-CPAM}). Then}
\begin{equation}
\label{eq:identity}
\int_{\R}U(x)\varphi(x)\,dx=\int_{\R}|\xi|^{-2s}\widehat{F}(\xi)\widehat{\varphi}(\xi)\,d\xi.
\end{equation}
It is easy to check that both integrals in \eqref{eq:identity} are absolutely convergent.
We start by proving the following identity:
$$
A_{-s}\int_{\R}\frac{\varphi(y)}{|y|^{1-2s}}\,dy =
\frac{1}{(2\pi)^{1/2}}\int_{\R}|\xi|^{-2s}\overline{\widehat{\varphi}(\xi)}\,d\xi,
$$
where $\overline{z}$ denotes the complex conjugate of $z\in\C$.
It is readily seen that both integrals above are absolutely convergent.
By Plancherel's identity and Fubini's Theorem,
\begin{align*}
\frac{1}{(2\pi)^{1/2}}\int_{\R}|\xi|^{-2s}\overline{\widehat{\varphi}(\xi)}\,d\xi &=
\frac{1}{\Gamma(s)(2\pi)^{1/2}} \int_0^\infty\int_{\R}e^{-t|\xi|^2} 
\overline{\widehat{\varphi}(\xi)}\,d\xi\,\frac{dt}{t^{1-s}} \\
&= \frac{1}{\Gamma(s)} \int_0^\infty \int_{\R} \frac{e^{-|y|^2/(4t)}}{(4\pi t)^{1/2}}
\varphi(y)\,dy\,\frac{dt}{t^{1-s}} \\
&= \int_{\R}\bigg[\frac{1}{\Gamma(s)}\int_0^\infty
\frac{e^{-|y|^2/(4t)}}{(4\pi t)^{1/2}}\,\frac{dt}{t^{1-s}}\bigg]\varphi(y)\,dy \\
&= A_{-s}\int_{\R}\frac{\varphi(y)}{|y|^{1-2s}}\,dy.
\end{align*}
Next, for any fixed $x\in\R$, by the properties of the Fourier transform,
$$
A_{-s}\int_{\R}\frac{\varphi(x-y)}{|y|^{1-2s}}\,dy = \int_{\R}|\xi|^{-2s}
\overline{\widehat{\varphi}(-\xi)}e^{ix\xi}\,d\xi.
$$
By multiplying both sides above by $F(x)$ and integrating in $dx$ we get
$$
A_{-s}\int_{\R}\int_{\R}\frac{F(x)\varphi(x-y)}{|y|^{1-2s}}\,dx\,dy
=\int_{\R}\int_{\R}|\xi|^{-2s}\overline{\widehat{\varphi}(-\xi)}F(x)e^{ix\xi}\,dx\,d\xi,
$$
which gives
$$
\int_{\R}U(x)\varphi(x)\,dx = A_{-s}\int_{\R}\varphi(x) \int_{\R}\frac{F(x-y)}{|y|^{1-2s}}\,dy\,dx
=\int_{\R}|\xi|^{-2s}\widehat{F}(-\xi)\overline{\widehat{\varphi}(-\xi)}\,d\xi.
$$

\item \textbf{We have $(-\Delta)^sU=F$ in the sense of distributions in $\mathcal{S}_s'$.} As $U$ is
bounded, we have that $U\in L_s(\R)$ (see \cite{Silvestre-CPAM} for the notation), namely,
$$
\int_{\R}\frac{|U(x)|}{1+|x|^{1+2s}}<\infty.
$$
Then $(-\Delta)^sU$ can be defined in the distributional sense:
for any function $\psi$ in the Schwartz class $\mathcal{S}$, we have
$\langle (-\Delta)^sU,\psi\rangle:=\langle U,(-\Delta)^s\psi\rangle$, see \cite[p.~73]{Silvestre-CPAM} for details.
The fractional Laplacian of $\psi\in\mathcal{S}$ is defined with the Fourier transform as
$$
\widehat{(-\Delta)^s\psi}(\xi)=|\xi|^{2s}\widehat{\psi}(\xi).
$$
Using the semigroup language and the Fourier transform as in \cite{Stinga, Stinga-Torrea} we get
$$
(-\Delta)^s\psi(x)=A_s\int_{\R}\frac{\psi(x)-\psi(y)}{|x-y|^{1+2s}}\,dy,
$$
where $A_s$ is as in~\eqref{eq:constante1d}.
We have that $\varphi:=(-\Delta)^s\psi\in\mathcal{S}_s$, namely,
$\varphi$ is a $C^\infty$ function such that $(1+|x|^{1+2s})D^k\varphi$
is bounded, for any $k\geq0$. The latter is claimed in \cite[p.~73]{Silvestre-CPAM}, but we
show it here for $k=0$
(the proof for $k\geq1$ is the same as the derivatives and the fractional Laplacian commute)
because we will need the computation at a later stage. Let us see that
\begin{equation}
\label{eq:cuenta}
\int_{\R}\frac{|\psi(x)-\psi(y)|}{|x-y|^{1+2s}}\,dy\leq C_{s,\psi}\frac{1}{|x|^{1+2s}},
\qquad\text{for all } |x|>1.
\end{equation}
Let $|x|>1$ and $y\in\R$. Suppose that $|x-y|<|x|/2$.
Then, for any intermediate point $\xi$ between $x$ and $y$, we have
$|x|\leq |x-\xi|+|\xi|\leq |x-y|+|\xi|\leq |x|/2+|\xi|$,
which gives $|x|\leq 2|\xi|$. As a consequence, by the Mean Value Theorem and using that $\psi\in\mathcal{S}$,
\begin{align*}
|\psi(x)-\psi(y)| &=(1+|\xi|)^3|\psi'(\xi)|\frac{|x-y|}{(1+|\xi|)^3} \\
&\leq C_\psi\frac{|x-y|}{(1+|x|)^3}\leq C_\psi\frac{|x-y|}{|x|^3}.
\end{align*}
From here,
\begin{align*}
\int_{|x-y|<|x|/2}\frac{|\psi(x)-\psi(y)|}{|x-y|^{1+2s}}\,dy &\leq
\frac{C_{\psi}}{|x|^3}\int_{|x-y|<|x|/2}\frac{|x-y|}{|x-y|^{1+2s}}\,dy \\
&=\frac{C_{s,\psi}}{|x|^{2+2s}}\leq\frac{C_{s,\psi}}{|x|^{1+2s}}.
\end{align*}
On the other hand,
\begin{align*}
\int_{|x-y|\geq|x|/2}\frac{|\psi(x)-\psi(y)|}{|x-y|^{1+2s}}\,dy
&\leq C_s\bigg(\frac{|\psi(x)|}{|x|^{2s}}+\frac{\|\psi\|_{L^1(\R)}}{|x|^{1+2s}}\bigg) \\
&= \frac{C_s}{|x|^{1+2s}}\big(|x\psi(x)|+\|\psi\|_{L^1(\R)}\big)\leq \frac{C_{s,\psi}}{|x|^{1+2s}}.
\end{align*}
Thus \eqref{eq:cuenta} is proved.
Let us finish then the proof of our claim (3). Using that $U\in L_s(\R)$, \eqref{eq:identity} and Plancherel's identity,
for any $\psi\in\mathcal{S}$,
\begin{equation}
\label{eq:distributions}
\begin{aligned}
	\langle (-\Delta)^sU,\psi\rangle &= \langle U,(-\Delta)^s\psi\rangle=\int_{\R}U(x)(-\Delta)^s\psi(x)\,dx \\
	&=\int_\R|\xi|^{-2s}\widehat{F}(\xi)\overline{\widehat{(-\Delta)^s\psi}(\xi)}\,d\xi \\
	&= \int_\R|\xi|^{-2s}\widehat{F}(\xi)\overline{|\xi|^{2s}\widehat{\psi}(\xi)}\,d\xi \\
	&=\int_\R\widehat{F}(\xi)\overline{\widehat{\psi}(\xi)}\,d\xi= \int_{\R}F(x)\psi(x)\,dx=\langle F,\psi\rangle.
\end{aligned}
\end{equation}

\item \textbf{$U$ is in $C^{0,\alpha+2s}$ and \eqref{eq:laregularidaddeU} holds.}
We showed in $(1)$ that $U$ is bounded. From \eqref{acotaciondeUporF},
\begin{equation}
\label{laprimera}
\|U\|_{L^\infty}\leq C_{s,R_0}\|F\|_{L^\infty}.
\end{equation}
Let $x_1,x_2\in\R$.
Suppose that $|x_1-x_2|\geq1$. Then, by using \eqref{laprimera},
\begin{equation}
\label{lasegunda}
|U(x_1)-U(x_2)| \leq 2\|U\|_{L^\infty}\leq C_{s,R_0}\|F\|_{L^\infty}|x_1-x_2|^{\alpha+2s}.
\end{equation}
Assume next that $|x_1-x_2|<1$. Let us take $r>0$ sufficiently large
so that $r>R_0+|x_1|+|x_2|$. As $F=0$ outside the interval $(-R_0,R_0)$ and $r>R_0$, we can write
\begin{equation}
\label{eq:U1U2}
\begin{aligned}
U(x_1)-U(x_2) &=A_{-s}\int_{-r}^r\big(F(y)-F(x_1)\big)\big(|x_1-y|^{-1+2s}-|x_2-y|^{-1+2s}\big)\,dy\\
&\quad+A_{-s}F(x_1)\int_{-r}^r\big(|x_1-y|^{-1+2s}-|x_2-y|^{-1+2s}\big)\,dy \\
&=A_{-s}\int_{-r}^r\big(F(y)-F(x_1)\big)\big(|x_1-y|^{-1+2s}-|x_2-y|^{-1+2s}\big)\,dy\\
&\quad+A_{-s}F(x_1)\big(N_{s,r}(x_1)-N_{s,r}(x_2)\big).
\end{aligned}
\end{equation}
Recall the expression for the derivative of the function
$N_{s,r}(x)$ for any $-r<x<r$ given in \eqref{N}.
In particular, we can use such a formula for any point between $x_1$ and $x_2$
because we have chosen $r$ large enough so that $-r<x_1,x_2<r$.
By the Mean Value Theorem, for some $\xi$ between $x_1$ and $x_2$, we have
\begin{align*}
|N_{s,r}(x_1)-N_{s,r}(x_2)| &= |N_{s,r}'(\xi)||x_1-x_2| \\
&\leq \bigg[\frac{1}{(\xi+r)^{1-2s}}+\frac{1}{(r-\xi)^{1-2s}}\bigg]|x_1-x_2|\to0,
\end{align*}
as $r\to\infty$.
Therefore, by taking the limit as $r\to\infty$ in \eqref{eq:U1U2}, we see that
$$
|U(x_1)-U(x_2)|\leq A_{-s}\int_{\R}|F(y)-F(x_1)|\big||x_1-y|^{-1+2s}-|x_2-y|^{-1+2s}\big|\,dy.
$$
The last integral is split into
$$
\int_\R~=\int_{|x_1-y|<2|x_1-x_2|}+\int_{|x_1-y|\geq2|x_1-x_2|}=:I+II.
$$
If $|y-x_1|<2|x_1-x_2|$ then $|y-x_2|\leq 4|x_1-x_2|$. Hence
\begin{align*}
I &\leq \int_{|x_1-y|<2|x_1-x_2|}\frac{|F(y)-F(x_1)|}{|x_1-y|^{1-2s}}\,dy
+\int_{|x_1-y|<2|x_1-x_2|}\frac{|F(y)-F(x_1)|}{|x_2-y|^{1-2s}}\,dy \\
&\leq [F]_{C^{0,\alpha}}\int_{|x_1-y|<2|x_1-x_2|}\frac{|y-x_1|^\alpha}{|x_1-y|^{1-2s}}\,dy \\
&\quad+C_\alpha[F]_{C^{0,\alpha}}\int_{|x_2-y|<4|x_1-x_2|}\frac{(|y-x_2|^\alpha+|x_2-x_1|^\alpha)}{|x_2-y|^{1-2s}}\,dy \\
&= C_{\alpha,s}[F]_{C^{0,\alpha}}|x_1-x_2|^{\alpha+2s}.
\end{align*}
To estimate the second integral, suppose that $|x_1-y|\geq 2|x_1-x_2|$. Let $\xi$ be an intermediate
point between $x_1$ and $x_2$. Then $|x_1-y|\leq|x_1-\xi|+|\xi-y|\leq|x_1-x_2|+|\xi-y|\leq
\frac{1}{2}|x_1-y|+|\xi-y|$. As a consequence, $|x_1-y|<2|\xi-y|$ and so
$|\xi-y|^{-2+2s}\leq C_s|x_1-y|^{-2+2s}$.
Using this, the Mean Value Theorem and the fact that $\alpha+2s<1$, we can estimate
\begin{align*}
II &\leq [F]_{C^{0,\alpha}}\int_{|x_1-y|\geq2|x_1-x_2|}|y-x_1|^\alpha
\big||x_1-y|^{-1+2s}-|x_2-y|^{-1+2s}\big|\,dy \\
&\leq C_s[F]_{C^{0,\alpha}}|x_1-x_2|\int_{|x_1-y|\geq2|x_1-x_2|}|y-x_1|^\alpha|x_1-y|^{-2+2s}\,dy \\
&=C_{\alpha,s}[F]_{C^{0,\alpha}}|x_1-x_2|^{\alpha+2s}.
\end{align*}
These estimates for $I$ and $II$, together with \eqref{laprimera} and \eqref{lasegunda},
imply~\eqref{eq:laregularidaddeU}.

\item \textbf{$(-\Delta)^sU$ is a well defined $C^{0,\alpha}$-function
and is given by the pointwise formula}
$$
(-\Delta)^sU(x)=A_s\int_{\R}\frac{U(x)-U(y)}{|x-y|^{1+2s}}\,dy,\qquad x\in\R.
$$
The pointwise formula follows from the results of \cite[p.~73]{Silvestre-CPAM}, see also~\cite{Stinga}.
Observe that the integral above is absolutely convergent and we have the estimate
\begin{equation}
\label{eq:acotacion}
\begin{aligned}
\int_{\R}\frac{|U(x)-U(y)|}{|x-y|^{1+2s}}\,dy&\leq [U]_{C^{0,\alpha+2s}}\int_{|x-y|<1}\frac{|x-y|^{\alpha+2s}}{|x-y|^{1+2s}}\,dy \\
&\quad+2\|U\|_{L^\infty}\int_{|x-y|\geq1}\frac{1}{|x-y|^{1+2s}}\,dy \\
&\leq C_{\alpha,s}\|U\|_{C^{0,\alpha+2s}}.
\end{aligned}
\end{equation}
This estimate is valid uniformly in $x\in\R$, hence $(-\Delta)^sU$ is bounded.
By \cite[p.~74]{Silvestre-CPAM}, see also \cite{Stinga}, we conclude that $(-\Delta)^sU$ is in $C^{0,\alpha}(\R)$.

\item \textbf{We have}
$$
(-\Delta)^sU(x)=F(x)\textbf{,}
$$
\textbf{in the pointwise sense.}
We have seen in (3) that $(-\Delta)^sU=F$ in the sense of distributions. In particular,
by the computation made in~\eqref{eq:distributions},
\begin{equation}
\label{eq:casi}
\int_{\R}U(x)(-\Delta)^s\psi(x)\,dx=\int_{\R}F(x)\psi(x)\,dx,
\end{equation}
for any $\psi\in\mathcal{S}$. If we show that
\begin{equation}
\label{eq:simetria}
\int_{\R}U(x)(-\Delta)^s\psi(x)\,dx=\int_{\R}\psi(x)(-\Delta)^sU(x)\,dx.
\end{equation}
then, by using \eqref{eq:casi}, we get
$(-\Delta)^sU(x)=F(x)$, for a.e. $x\in\R$, and, by continuity, $(-\Delta)^sU(x)=F(x)$, for every $x\in\R$.
So we are left to show \eqref{eq:simetria}.
Since $(-\Delta)^sU$ is bounded (see (5)) the integral in the right hand side of \eqref{eq:simetria}
is absolutely convergent. We write the left hand side of \eqref{eq:simetria} as
\begin{equation}
\label{1}
\int_{\R}U(x)(-\Delta)^s\psi(x)\,dx=A_s\int_{\R}\int_{\R}U(x)\frac{\psi(x)-\psi(y)}{|x-y|^{1+2s}}\,dy\,dx,
\end{equation}
and
\begin{equation}
\label{2}
\begin{aligned}
\int_{\R}U(y)(-\Delta)^s\psi(y)\,dy 
&=A_s\int_{\R}\int_{\R}U(y)\frac{\psi(y)-\psi(x)}{|y-x|^{1+2s}}\,dx\,dy \\
&=-A_s\int_{\R}\int_{\R}U(y)\frac{\psi(x)-\psi(y)}{|x-y|^{1+2s}}\,dy\,dx.
\end{aligned}
\end{equation}
In the second identity in \eqref{2} we applied Fubini's Theorem.
To justify it, observe that \eqref{eq:cuenta} and \eqref{eq:acotacion} hold for $\psi\in\mathcal{S}$, so
\begin{align*}
\int_{\R}|U(y)|\int_{\R}\frac{|\psi(y)-\psi(x)|}{|y-x|^{1+2s}}\,dx\,dy &\leq
C_{s,\psi}\bigg(\int_{|y|\leq1}|U(y)|\,dy+\int_{|y|>1}\frac{|U(y)|}{|y|^{1+2s}}\,dy\bigg) \\
&\leq C_{s,\psi}\|U\|_{L^\infty}<\infty.
\end{align*}
By adding \eqref{1} and \eqref{2},
$$
\int_{\R}U(x)(-\Delta)^s\psi(x)\,dx 
= \frac{A_s}{2}\int_{\R}\int_{\R}\frac{(U(x)-U(y))(\psi(x)-\psi(y))}{|x-y|^{1+2s}}
\,dy\,dx.
$$
To show that the right hand side of \eqref{eq:simetria} is also equal to the double integral above
we only need to verify that Fubini's Theorem can be applied in \eqref{2} with $U$ and $\psi$ interchanged.
But this is simpler now because of \eqref{eq:acotacion}:
$$
\int_{\R}|\psi(y)|\int_{\R}\frac{|U(y)-U(x)|}{|y-x|^{1+2s}}\,dx\,dy\leq
C_{\alpha,s}\|U\|_{C^{0,\alpha+2s}}\int_{\R}|\psi(y)|\,dy<\infty.
$$
Thus \eqref{eq:simetria} is proved.

\item \textbf{$U$ is the unique classical solution to $(-\Delta)^sU=F$ in $\R$ which
vanishes at infinity.} The previous items show that $U$ is a classical solution vanishing at infinity.
Let $V$ be another classical solution vanishing at infinity. Then the difference $W:=U-V$ satisfies
$$
\begin{cases}
(-\Delta)^sW=0, &\text{in } \R, \\
|W(x)|\to0, &\text{as } |x|\to\infty, \\
W\in L^\infty(\R).
\end{cases}
$$
By the Liouville Theorem for the fractional Laplacian (see for example
\cite{Landkof}), $W$ is a constant and, since it vanishes at
infinity, $W=0$.
\end{enumerate}
\end{proof}

\begin{rem}
It is worth noticing that in \cite[p.~117]{StSingular} identity \eqref{eq:identity} is shown for $\varphi\in\mathcal{S}$
by using spherical harmonics.
Instead, our proof is more elementary (and valid for more general functions~$\varphi$)
in the sense we only use the Gamma function and Plancherel's identity.
\end{rem}

\noindent\textbf{Acknowledgments.} We would like to thank the referee 
for detailed comments and suggestions that helped us to improve the 
presentation of the paper.



\end{document}